\documentclass{article}
\title{Semicoverings: a generalization of covering space theory}
\author{Jeremy Brazas}
\usepackage{stmaryrd,graphicx}
\usepackage{amssymb,mathrsfs,amsmath,amscd,fullpage,amsthm}
\usepackage{float}
\usepackage[all,cmtip]{xy}
\DeclareMathAlphabet{\mathpzc}{OT1}{pzc}{m}{it}
\usepackage{amsfonts,txfonts,pxfonts,latexsym,wasysym}
\newcommand{\sus}{\Sigma(X_{+})}

\newcommand{\ra}{\to}
\newcommand{\bs}{\backslash}

\newcommand{\sG}{\mathcal{G}}
\newcommand{\sH}{\mathcal{H}}

\newcommand{\txf}{\tilde{X}_{F}}
\newcommand{\pf}{p_{F}}
\newcommand{\tf}{\Theta_{F}}

\newcommand{\pitx}{\pi_{1}^{\tau}(X,x_0)}

\newcommand{\px}{\mathcal{P}X}
\newcommand{\py}{\mathcal{P}Y}
\newcommand{\pxxo}{(\mathcal{P}X)_{x_0}}
\newcommand{\pyyo}{(\mathcal{P}Y)_{y_0}}
\newcommand{\pxx}{(\mathcal{P}X)_{x}}

\newcommand{\pyy}{(\mathcal{P}Y)_{y}}

\newcommand{\qtopgrpd}{\mathbf{qTopGrpd}}

\newcommand{\topgrpd}{\mathbf{TopGrpd}}
\newcommand{\qtg}{\mathbf{qTopGrp}}
\newcommand{\set}{\mathbf{Set}}

\newcommand{\tg}{\mathbf{TopGrp}}

\newcommand{\spaces}{\mathbf{Top}}
\newcommand{\grp}{\mathbf{Grp}}

\newcommand{\bspaces}{\mathbf{Top_{\ast}}}

\newcommand{\stc}{wep-connected}
\newcommand{\estc}{locally wep-connected}
\newcommand{\ee}{locally well-ended}
\newcommand{\we}{well-ended}
\newcommand{\wt}{well-targeted}
\newcommand{\et}{locally well-targeted}

\newcommand{\pit}{\pi_{1}^{\tau}}
\newcommand{\scovx}{\mathbf{SCov}(X)}
\newcommand{\cscovx}{\mathbf{SCov_{0}}(X)}
\newcommand{\covx}{\mathbf{Cov}(X)}
\newcommand{\ccovx}{\mathbf{Cov_{0}}(X)}

\newcommand{\pyx}{p\colon Y\to X}
\newtheorem{theorem}{Theorem}[section]
\newtheorem{lemma}[theorem]{Lemma}
\newtheorem{proposition}[theorem]{Proposition}
\newtheorem{corollary}[theorem]{Corollary}
\newtheorem{definition}[theorem]{Definition}

\newtheorem{example}[theorem]{Example}

\newtheorem{remark}[theorem]{Remark}

\newtheorem{approximation}[theorem]{Approximation of $\tau(\mathcal{G})$}
\newtheorem{pathlifting}[theorem]{Canonical lifts of paths}
\newtheorem{homotopylifting}[theorem]{Canonical lifts of homotopies}

\newtheorem{basis}[theorem]{A basis for the topology of $\txf$}
\begin{document}
\maketitle
\begin{abstract}
The fundamental groupoid of a space becomes enriched over the category of topological spaces when the hom-sets are endowed with topologies intimately related to universal constructions of topological groups. This paper is devoted to a generalization of classical covering theory in the context of this construction.
\end{abstract}
\section{Introduction}
This paper is the sequel to \cite{Br10.2} where the fundamental group of a space is endowed with the finest group topology such that the natural map from the loop space identifying homotopy classes is continuous. The resulting homotopy invariant $\pi_{1}^{\tau}$ takes values in the category of topological groups and can distinguish spaces with isomorphic fundamental groups. Additionally, there is a natural connection to free topological groups and free topological products unlikely to arise from the more historical shape theoretic approach to spaces with complicated local structure. The present paper includes the theory of \textit{semicoverings}, a generalization of covering space theory in the context of a topologically enriched fundamental groupoid $\pi^{\tau}$ whose vertex groups are given by $\pi_{1}^{\tau}$.

Generalizations of the notion of covering map and extensions of covering theoretic techniques to spaces beyond those in the classical theory (path connected, locally path connected and semilocally 1-connected spaces) have appeared in many different contexts \cite{BS98,BDLM08,FZ07,Fox74,HP98,Lub62}. We emphasize the assessment in the introduction of \cite{FZ07} that the properties one should require of a ``generalized covering map" depend on the intended application. A semicovering map $\pyx$ is a local homeomorphism such that whenever $p(y)=x$ and $f$ is a path or homotopy of paths starting at $x$, there is a unique lift $\tilde{f}_{y}$ of $f$ starting at $y$. Moreover, we demand that each lifting assignment $f\mapsto \tilde{f}_{y}$ for paths and homotopies is continuous with respect to the compact-open topology on function spaces. Except for local triviality, semicoverings enjoy nearly all of the important properties of coverings; in this sense the notion of semicovering is quite close to the notion of covering. The intentions of this generalization are to further study the topology on $\pi_{1}$ introduced in \cite{Br10.2}, identify a suitable classification of generalized coverings applicable to spaces with non-trivial local structure, and develop geometric tools for studying topological groups much like those used in applications of covering space theory to group theory \cite[Ch. 14]{Hig71}.

Covering-type theories applicable to non-locally path connected spaces are less prevalent due to the existence of troubling examples such as Zeeman's example \cite[Example 6.6.14]{HW60} of a planar set which admits non-equivalent coverings that, under the usual classification, do not give the same conjugacy class of subgroup of the fundamental group. Lubkin's theory of coverings \cite{Lub62} overcomes this obstacle using a more general notion of ``space" and ``group." The authors of \cite{ATHP99,HP98} achieve a quite general theory by attaching extra data (equivalence classes of locally constant presheaves) to their projection maps and providing a classification in terms of the fundamental pro-groupoid. Conveniently, semicoverings of a space are genuine maps of topological spaces and are classified for a broad class of topological spaces. In particular, the full theory of semicoverings applies to the class of locally wep-connected spaces defined in Section \ref{localpropertiesofendpointsofpaths}. This class contains all locally path connected spaces and ``enough" non-locally path connected spaces. By ``enough" we mean that any topological group is realized as the fundamental group of a locally wep-connected space and for any space $X$, $\pitx$ may be approximated up to induced isomorphism on $\pit$ by a locally wep-connected space. Thus for application of semicoverings to the theory of topological groups, it suffices to study locally wep-connected spaces.

Recall that for a path connected, locally path connected, and semilocally 1-connected space $X$, the category $\covx$ of coverings of $X$ is naturally equivalent to the category $\mathbf{CovMor}(\pi X)$ of covering morphisms of the fundamental groupoid $\pi X$ \cite[10.6.1]{Brown06}. Additionally, $\mathbf{CovMor}(\pi X)$ is equivalent to the functor category $\mathbf{Func}(\pi X,\set)$. The resulting equivalence $\covx\ra \mathbf{Func}(\pi X,\set)$ is an isomorphism of categories given by taking a covering to its monodromy. The classification of semicoverings is also stated as an equivalence of categories but requires the language of enriched category theory \cite{KellyEnriched}.

A $\spaces$-category is a category $\mathcal{C}$ enriched over $\spaces$, the category of topological spaces, in the sense that each hom-set is equipped with a topology such that all composition maps are continuous. For instance, we view $\set$ as a $\spaces$-category by viewing each set as a discrete space and giving sets of functions the topology of point-wise convergence. When the underlying category of $\mathcal{C}$ is a groupoid and each inversion map $\mathcal{C}(a,b)\ra \mathcal{C}(b,a)$ is continuous, $\mathcal{C}$ is a $\spaces$-groupoid. In section \ref{topgroupoidsection} we construct, for each space $X$, a $\spaces$-groupoid $\pit X$ whose underlying groupoid is the fundamental groupoid $\pi X$. Given $\spaces$-categories $\mathcal{A}$ and $\mathcal{B}$, a $\spaces$-functor is a functor $F\colon\mathcal{A}\to \mathcal{B}$ such that each function $\mathcal{A}(a_1,a_2)\ra \mathcal{B}(F(a_1),F(a_2))$ is continuous. We say $F$ is \textit{open} if each map $\mathcal{A}(a_1,a_2)\ra \mathcal{B}(F(a_1),F(a_2))$ is also open. A $\spaces$-natural transformation of $\spaces$-functors is a natural transformation of the underlying functors. The category of $\spaces$-functors $\mathcal{A}\to \mathcal{B}$ and $\spaces$-natural transformations is denoted $\mathbf{TopFunc}(\mathcal{A},\mathcal{B})$.

Our main result (Theorem \ref{classificationI}) states that for any locally wep-connected space, the category $\scovx$ of semicoverings is naturally isomorphic as a category to the category of enriched functors $\pit X\ra \set$. An analogue of the equivalence $\covx\simeq \mathbf{CovMor}(\pi X)$ is obtained after one observes that $\mathbf{TopFunc}(\pi^{\tau} X,\set)$ is naturally equivalent to the category $\mathbf{OCovMor}(\pi^{\tau}X)$ of open covering morphisms $\sG\ra \pi^{\tau}X$ of $\spaces$-groupoids. This classification also restricts to a classification of connected semicoverings of $X$ in terms of continuous, transitive actions of the topological group $\pitx$ on discrete spaces. These results include, as special cases, classical covering theory, Spanier's extension to non-semilocally 1-connected spaces via ``Spanier groups" $\pi(\mathscr{U},x_0)$ \cite{Spanier66,FRVZ11}, and Fox's fundamental theorem of overlays \cite{Fox74,Koc90}.

The author thanks Maria Basterra for many helpful conversations.
%and the referee for the suggestion of stating the classification of semicoverings as an equivalence of categories.
%\footnote{This research was partially supported by a Dissertation Year Fellowship from the University of New Hampshire Graduate School.}

\section{Continuous lifting of paths and homotopies}
In general, mapping spaces $\spaces(X,Y)$ will be given the compact-open topology generated by subbasis sets $\langle K,U\rangle=\{f|f(K)\subseteq U\}$, $K\subseteq X$ compact, $U\subseteq Y$ open. Let $\px$ denote the space of paths $\alpha\colon I=[0,1]\ra X$ and $c_{x}$ denote the constant path at $x\in X$. If $\mathscr{B}$ is a basis for the topology of $X$ which is closed under finite intersection, sets of the form $\bigcap_{j=1}^{n}\langle K_{n}^{j},U_j\rangle$ where $K_{n}^{j}=\left[\frac{j-1}{n},\frac{j}{n}\right]$ and $U_j\in \mathscr{B}$ form a convenient basis for the topology of $\px$.

For any fixed, closed subinterval $A\subseteq I$, let $T_A\colon I\ra A$ be the unique, increasing, linear homeomorphism. For a path $\alpha\in\px$, $\alpha_{A}=\alpha|_{A}\circ T_{A}\colon I\ra A\ra X$ is the restricted path of $p$ to $A$. As a convention, if $A=\{t\}\subseteq I$, let $\alpha_{A}=c_{\alpha(t)}$. Note that if $0=t_0\leq t_1\leq ...\leq t_n=1$, knowing the paths $\alpha_{[t_{i-1},t_i]}$ for $i=1,...,n$ uniquely determines $\alpha$. It is simple to describe concatenations of paths with this notation: If $\alpha_1,\alpha_2,...,\alpha_n\in\px$ such that $\alpha_j(1)=\alpha_{j+1}(0)$ for each $j=1,...,n-1$, the \textit{n-fold concatenation} of this sequence is the unique path $\beta=\alpha_1\ast \alpha_2\ast \dots\ast \alpha_n$ such that $\beta_{K_{n}^{j}}=\alpha_j$ for each $j=1,...,n$. It is an elementary fact of the compact-open topology that concatenation $\px\times_{X}\px=\{(\alpha,\beta)|\alpha(1)=\beta(0)\}\ra \px$, $(\alpha,\beta)\mapsto \alpha\ast\beta$ is continuous. If $\alpha\in \px$, then $\alpha^{-1}(t)=\alpha(1-t)$ is the \textit{reverse} of $\alpha$ and for a set $A\subseteq \px$, $A^{-1}=\{\alpha^{-1}|\alpha\in A\}$. The operation $\alpha\mapsto \alpha^{-1}$ is a self-homeomorphism of $\px$.

If $x\in X$, let $\pxx=\{\alpha\in \px|\alpha(0)=x\}$, $(\mathcal{P}X)^{y}=\{\alpha\in \px|\alpha(1)=y\}$, and $\px(x,y)=\pxx\cap (\px)^{y}$ be subspaces of $\pxx$. In notation, we do not distinguish a neighborhood from being an open set in $\px$ or any of its subspaces but rather leave this to context.  We also use the notation $\Omega(X,x)=\px(x,x)$. Each of these constructions gives either a functor $\spaces\ra \spaces$ or $\bspaces\ra \bspaces$. For instance, $(X,x)\mapsto (\pxx,c_x)$ is a functor which is $\mathcal{P}f(\alpha)=f\circ \alpha$ on morphisms.
\begin{definition} \emph{
A map $\pyx$ has \textit{continuous lifting of paths} if $\mathcal{P}p\colon\pyy\ra (\px)_{p(y)}$ is a homeomorphism for each $y\in Y$. }
\end{definition}

Certainly every map with continuous lifting of paths has unique path lifting since this property is equivalent to the injectivity of $\mathcal{P}p$ in the above definition. Consequently, a map with continuous lifting of paths has the unique lifting property with respect to all path connected spaces \cite[2.2 Lemma 4]{Spanier66}. The condition that $\mathcal{P}p$ be a homeomorphism is much stronger than the existence and uniqueness of lifts of paths since each inverse $L_p\colon(\mathcal{P}X)_{p(y)}\ra \pyy$, which we refer to as the \textit{lifting homeomorphism}, taking a path $\alpha$ to the unique lift $\tilde{\alpha}_{y}$ starting at $y$ is required to be continuous.

Let $\Delta_2=\left\{(s,t)\in I^{2}|s+t\leq 1\right\}$ be the 2-simplex with edges $e_1,e_2,e_3$ opposite to vertices $(1,0),(0,1),(0,0)$ respectively. Let $\partial_j\colon e_j\hookrightarrow\Delta_2$ denote each inclusion and $(\Phi X)_{x}$ be the space of relative maps $(\Delta_2,e_{1})\to(X,\{x\})$. If $\alpha,\beta \in \pxx$ such that $\alpha(1)=\beta(1)$, then $\alpha\simeq \beta$ (rel. endpoints) if and only if there is a $\phi\in (\Phi X)_{x}$ such that $\alpha$ is $\xymatrix{I \ar[r]^-{\partial_{2}} & \Delta_2 \ar[r]^-{\phi} & X}$ and $\beta$ is $\xymatrix{I \ar[r]^-{\cong} & e_{3} \ar[r]^-{\partial_{3}} & \Delta_2 \ar[r]^-{\phi} & X}$ (the first map is the inverse of the homeomorphic projection of $e_3$ onto $e_2=I$). Thus elements of $(\Phi X)_{x}$ are homotopies of paths. Just as with path spaces, $\Phi\colon\bspaces\ra \bspaces$ is a functor which is $\Phi f(\phi)=f\circ \phi$ on morphisms.
\begin{definition}  \emph{
A map $\pyx$ has \textit{continuous lifting of homotopies} if $\Phi p\colon(\Phi Y)_{y}\ra (\Phi X)_{p(y)}$ is a homeomorphism for each $y\in Y$.}
\end{definition}
Let $\sG$ be a groupoid with source and target maps $s,t\colon\sG\ra Ob(\sG)$. The star of $\sG$ at $x\in Ob(\sG)$ is $\sG_{x}=\{g\in \sG|s(g)=x\}$. Additionally, $\sG^{y}=\{g\in \sG|t(g)=y\}$, $\sG(x,y)=\sG_{x}\cap \sG^{y}$ and $\sG(x)=\sG(x,x)$. If $g\in \sG(x,y)$, right and left multiplication by $g$ is $\rho_{g}\colon\sG(w,x)\ra \sG(w,y)$ and $\lambda_{g}\colon\sG(y,z)\ra \sG(x,z)$ respectively. A \textit{covering morphism} is a functor $F\colon\sG\ra \sG '$ of groupoids such that each function $\sG_{x}\ra \sG '_{F(x)}$, $g\mapsto Fg$ is a bijection. If $g\in \sG '_{F(x)}$, then $\tilde{g}_{x}$ denotes the unique $\tilde{g}_{x}\in \sG_{x}$ such that $F(\tilde{g}_{x})=g$. We refer to \cite{Brown06} for the theory of covering morphisms, groupoid actions, and the fundamental groupoid. The following lemma is straightforward given the above definitions.
\begin{lemma} \label{contliftingprops}
Let $\pyx$ be a map with continuous lifting of paths and homotopies.
\begin{enumerate}
\item $p$ induces a covering morphism $\pi p\colon\pi Y\ra \pi X$ of fundamental groupoids.
\item $Y$ is a right $\pi X$-set via the action $p^{-1}(x_1)\times \pi X(x_1,x_2)\ra p^{-1}(x_2)$, $(y,[\alpha])\mapsto y\cdot [\alpha]=\tilde{\alpha}_{y}(1)$.
\item If $p(y_i)=x_i$, $i=1,2$ and $\beta\in \px(x_1,x_2)$, then $[\beta]\in Im \left(\pi p\colon\pi Y(y_1,y_2)\ra \pi X(x_1,x_2)\right)$ if and only if $y_{1}\cdot[\beta]=y_2$.
\end{enumerate}
\end{lemma}
Thus a map with continuous lifting of paths and homotopies has monodromy in the following sense.
\begin{definition} \emph{
The \textit{monodromy} of a map $\pyx$ with continuous lifting of paths and homotopies is the functor $\mathscr{M}p\colon\pi X\ra \set$ which takes a point $x\in X$ to the fiber $p^{-1}(x)$ and a class $[\alpha]\in \pi X(x_1,x_2)$ to the function $p^{-1}(x_1)\ra p^{-1}(x_2)$, $y\mapsto y\cdot [\alpha]$. }
\end{definition}
The following useful lemma is a generalization of a well-known lifting result from covering space theory (\cite[10.5.3]{Brown06}) and illustrates precisely why the notion of continuous lifting is worth considering in covering space theory and its generalizations.
\begin{lemma} \label{generalliftinglemma}
Let $\pyx$ be a map with continuous lifting of paths and homotopies and $W$ be a space whose path components are open and such that for each $w\in W$ evaluation $ev_{1}\colon(\mathcal{P}W)_{w}\ra W$, $\beta\mapsto \beta(1)$ is quotient onto the path component of $w$. If $f\colon W\ra X$ is any map such that $\pi f\colon\pi W\ra \pi X$ lifts to a morphism $\Psi\colon\pi W\ra \pi Y$ of groupoids (i.e. $\pi p\circ \Psi= \pi f$), then $\tilde{f}=Ob(\Psi)\colon W\ra Y$ is continuous and $\Psi=\pi\tilde{f}$.
\end{lemma}
\begin{proof}
Since the path components of $W$ are open, it suffices to show the restriction of $\tilde{f}$ to each path component of $W$ is open. Thus it suffices to the prove the lemma for $W$ path connected. Suppose $\tilde{f}(w_0)=y_0$ and $p(y_0)=x_0=f(w_0)$. Note that $\tilde{f}$ is determined as follows: given $w\in W$ and any path $\beta\in (\mathcal{P}W)_{w_0}$ such that $\alpha(1)=w$, $\tilde{f}(w)=\widetilde{f\circ\alpha}_{y_0}(1)$. That this description of $\tilde{f}(w)$ does not depend on the choice of $w_0$ or $\beta\in (\mathcal{P}W)_{w_0}$ follows, in the usual manner, from the unique path lifting of $p$ and the assumption that $\Psi$ is a lift of $\pi f$. By functorality, $\mathcal{P}f\colon(\mathcal{P}W)_{w_0}\ra (\mathcal{P}X)_{x_0}$ is continuous and since $p$ has continuous lifting of paths, there is a lifting homeomorphism $L_p\colon\pxxo\ra \pyyo$. The diagram \[\xymatrix{ (\mathcal{P}W)_{w_0} \ar[d]_-{ev_1} \ar[r]^-{\mathcal{P}f} & \pxxo \ar@/^1pc/[rr]^-{id}\ar[r]_-{L_P} & \pyyo \ar[d]^-{ev_1} \ar[r]_-{\mathcal{P}p} & \pxxo \ar[d]^-{ev_1}\\
W  \ar@/_1pc/[rrr]_-{f} \ar@{-->}[rr]^-{\tilde{f}} && Y \ar[r]^-{p} & X}.\]commutes and since the left-most vertical map is quotient, $\tilde{f}$ is continuous by the universal property of quotient spaces. Since $\pi p$ is a covering morphism and $\Psi$ and $\pi\tilde{f}$ are both lifts of $\pi f$, we have $\Psi=\pi\tilde{f}$.
\end{proof}
\section{Semicovering maps}
\begin{definition}\label{semicoveringdef} \emph{
A \textit{semicovering map} $\pyx$ is a local homeomorphism with continuous lifting of paths and homotopies.}
\end{definition}
We refer to $Y$ as a \textit{semicovering space} of $X$ and often refer to $p$ simply as a \textit{semicovering} of $X$. If $q\colon Y '\ra X$ is another semicovering of $X$, a morphism of semicoverings is a map $f\colon Y \ra Y'$ such that \[\xymatrix{ Y \ar[rr]^{f} \ar[dr]_{p} & & Y ' \ar[dl]^{p '} \\ & X }\] commutes. This defines a category $\scovx$ of semicoverings of $X$. Two semicoverings of $X$ are then \textit{equivalent} if they are isomorphic in this category. A semicovering $\pyx$ is \textit{connected} if $Y$ is non-empty and path connected. Let $\cscovx$ denote the full subcategory of connected semicoverings. A \textit{universal semicovering} of $X$ is a semicovering initial in $\cscovx$.
\begin{remark} \label{emptysemicov} \emph{
Every zero morphism $\emptyset\ra X$ is vacuously a semicovering called the \textit{empty semicovering of }$X$. The monodromy of the empty semicovering is the unique (provided $X$ is path connected) functor $\pi X\ra \set$ whose value on each object is the emptyset. Note that if $y\in Y$ and $x\in X$, any path $\alpha\in \px(p(y),x)$ lifts to a path $\tilde{\alpha}_{y}$ such that $p(\tilde{\alpha}_{y}(1))=x$. Thus if $X$ is path connected and $\pyx$ is a semicovering map, then either $Y=\emptyset$ or $p$ is surjective.}
\end{remark}
\begin{remark} \emph{
It is evident from Lemma \ref{generalliftinglemma} that a semicovering has the homotopy lifting property with respect to locally path connected, simply connected spaces and is therefore a Serre fibration with discrete fibers.}
\end{remark}

It is well-known that if one does not restrict to spaces with universal coverings the composition of two (connected) covering maps is not always a covering map. On the other hand, it is straightforward from the definition of semicovering that the composition of two semicoverings is a semicovering. Thus connected semicoverings have the desirable ``two out of three" property. The following lemma is an exercise in point-set topology.
\begin{lemma} \label{lochomeomorphism}
Let $p\colon X\ra Y$, $q\colon Y\ra Z$ and $r=q\circ p$ be surjective maps. If two of $p,q,r$ are local homeomorphisms, then so is the third. If two of $p,q,r$ have continuous lifting of paths and homotopies, then so does the third.
\end{lemma}
\begin{corollary} \label{twoofthree}
Let $p\colon X\ra Y$, $q\colon Y\ra Z$, and $r=q\circ p$ be maps where $Y$ and $Z$ are path connected. If two of $p,q,r$ are semicoverings, so is the third.
\end{corollary}
\begin{proof}
The case where $X=\emptyset$ is trivial. Suppose then that $p(x_0)=y_0$ and $q(y_0)=z_0=r(x_0)$. If $p$ and $q$ are semicoverings, then $r$ clearly satisfies the conditions in Definition \ref{semicoveringdef}. If $q$ and $r$ are semicoverings and $y\in Y$ take $\alpha\in (\py)_{y_0}$ with $\alpha(1)=y$. Then $q\circ \alpha\in (\mathcal{P}Z)_{z_0}$ has unique lift $\widetilde{q\circ \alpha}_{x_0}\in (\px)_{x_0}$ with endpoint $x=\widetilde{q\circ \alpha}_{x_0}(1)$. Since $q\circ p\circ \widetilde{q\circ \alpha}_{x_0}=q\circ \alpha$ and $q$ has unique path lifting, we have $p\circ \widetilde{q\circ \alpha}_{x_0}=\alpha$. Therefore $p(x)=\alpha(1)=y$ and $p$ is surjective. By Lemma \ref{lochomeomorphism}, $p$ is a semicovering. Lastly, suppose $p$ and $r$ are semicoverings. Since $Z$ is path connected, $q$ is surjective and therefore a semicovering by Lemma \ref{lochomeomorphism}.
\end{proof}
We now check that every covering is a semicovering.
\begin{remark} \emph{
If $\pyx$ is a covering map and $U\subseteq X$ is an evenly covered neighborhood, then $p^{-1}(U)$ is the disjoint union $\coprod_{\lambda}V_{\lambda}$ of \textit{slices} $V_{\lambda}$ over $U$. The collection of slices over evenly covered neighborhoods form a basis $\mathscr{B}_{p}$ for the topology of $Y$ which is closed under finite intersection. Consequently, the neighborhoods of the form $\bigcap_{j=1}^{n}\langle K_{n}^{j},V_j\rangle$, $V_{j}\in \mathscr{B}_{p}$ give a basis for the topology of $\py$.}
\end{remark}
\begin{proposition}
For any space $X$, $\covx$ and $\ccovx$ are full subcategories of $\scovx$ and $\ccovx$ respectively. 
\end{proposition}
\begin{proof}
Certainly, a covering map $\pyx$ is a local homeomorphism. Suppose $p(y_0)=x_0$. Since covering maps uniquely lift paths and homotopies of paths, $\mathcal{P}p\colon\pyyo\ra \pxxo$ and $\Phi p\colon(\Phi Y)_{y_0}\ra(\Phi X)_{x_0}$ are bijective. Both are continuous by functorality. Let $\mathcal{U}=\bigcap_{j=1}^{n}\langle K_{n}^{j},U_j\rangle$ be a basic non-empty open neighborhood in $\pyyo$ where each $U_j\in \mathscr{B}_{p}$. Since $\mathcal{U}$ is non-empty, there is a path $\tilde{\alpha}_{y_0}\in \mathcal{U}$ that is the lift of \[\alpha=p\circ \tilde{\alpha}_{y_0}\in \mathcal{V}=\left(\bigcap_{j=1}^{n}\left\langle K_{n}^{j},p(U_j)\right\rangle\right)\cap \left(\bigcap_{j=1}^{n-1}\left\langle \left\{\frac{j}{n}\right\},p(U_{j}\cap U_{j+1})\right\rangle\right)\subseteq \pxxo.\]Clearly $\mathcal{P}p(\mathcal{U})\subseteq\mathcal{V}$. The lift $\tilde{\alpha}_{y_0}$ has the following description: There are homeomorphisms $h_j\colon p(U_j)\ra U_j$ such that $p\circ h_j$ is the identity of $p(U_j)$. For each $t\in K_{n}^{j}$, we have $\tilde{\alpha}_{y_0}(t)=h_j\circ \alpha(t)$. Note that if $\beta$ is any other path in $\mathcal{V}$, the unique lift $\tilde{\beta}_{y_0}\in \pyyo$ is defined in the same way, that is, for each $t\in K_{n}^{j}$, $\tilde{\beta}_{y_0}(t)=h_j\circ \beta(t)$. The equality $\mathcal{P}p(\mathcal{U})=\mathcal{V}$ implies that $\mathcal{P}p$ is open. A completely analogous argument may be used to show that $\Phi p\colon(\Phi Y)_{y_0}\ra (\Phi X)_{x_0}$ is a homeomorphism. Recalling how lifts of homotopies of paths are constructed, one may proceed by viewing $\Delta_2$ as a simplicial complex and taking a basic open neighborhood of a homotopy $\tilde{G}_{y_0}\in(\Phi Y)_{y_0}$ to be of the form $\mathcal{U}=\bigcap_{\sigma\in sd_{n}(\Delta_2)}\langle \sigma,U_{\sigma}\rangle$ where the intersection is taken over 2-simplices $\sigma$ in the n-th barycentric subdivision $sd_{n}(\Delta_2)$ of $\Delta_2$ and $U_{\sigma}\in \mathscr{B}_{p}$. Then \[\mathcal{V}=\left(\bigcap_{\sigma\in sd_{n}(\Delta_2)}\left\langle \sigma,p(U_{\sigma})\right\rangle\right)\cap \left(\bigcap_{e=\sigma\cap \sigma '}\left\langle e,p(U_{\sigma}\cap U_{\sigma '})\right\rangle\right)\]is an open neighborhood of $G=p\circ \tilde{G}_{y_0}$ satisfying $\Phi p(\mathcal{U})=\mathcal{V}$. Here the second intersection ranges over all 1-simplices which are the intersection of two 2-simplices in $sd_{n}(\Delta_2)$. Thus every covering is a semicovering and $\covx$ and $\ccovx$ are full subcategories.
\end{proof}
\begin{example}\emph{
Since every covering is a semicovering, the composition of covering maps is always a semicovering map even if it is not a covering map. This fact alone provides simple examples of semicoverings which are not coverings. In the non-connected case, it is easy to see that $S^1\times \{1,2,...\}\ra S^1$, $(z,n)\ra z^n$ is a semicovering but not a covering. There are connected semicoverings of the Hawaiian earring \[\mathbb{HE}=\bigcup_{n\geq 1}\left\{(x,y)\in\mathbb{R}^{2}|\left(x-\frac{1}{n}\right)^{2}+y^2=\frac{1}{n^2}\right\}\] which are not coverings; variations of the semicovering space which is the answer to Exercise 6 in Chapter 1.3 of Hatcher \cite{Hatcher02} illustrate the extensiveness of semicoverings of non-semilocally 1-connected spaces beyond coverings.}
 \begin{figure}[H] \centering \includegraphics[height=2.5in]{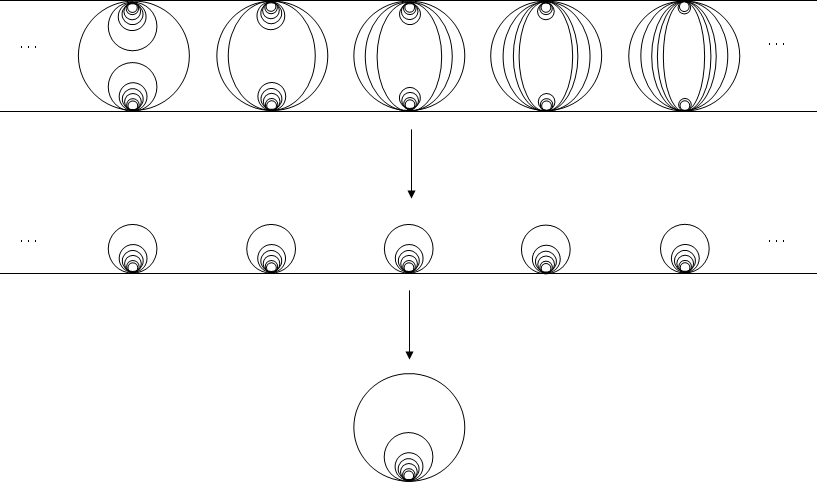}
\caption{A two-sheeted covering of a covering of $\mathbb{HE}$ which is a semicovering of $\mathbb{HE}$ but fails to be a covering of $\mathbb{HE}$.}
\end{figure}
\end{example}
It is also worthwhile to note that semicoverings admit an adequate theory of pullbacks.
\begin{proposition} \label{pullbacks}
If $\pyx$ is a semicovering of $X$, $f\colon W\ra X$ is a map, and $W\times_{X}Y=\{(w,y)|f(w)=p(y)\}$ is the pullback, the projection $f^{\ast}p\colon W\times_{X}Y\ra W$ is a semicovering of $W$. Consequently, $\mathbf{Scov}(X)$ is contravariant in $X$ with values in the category of small categories.
\end{proposition}
\begin{proof}
If $(w,y)\in W\times_{X}Y$, let $U$ be an open neighborhood of $y$ in $Y$ which is mapped homeomorphically onto $p(U)$ and $V=f^{-1}(p(U))\subseteq W$. Now $(V\times U)\cap W\times_{X}Y$ is an open neighborhood of $(w,y)$ in $W\times_{X}Y$ mapped homeomorphically onto $V$ by $f^{\ast}p$. Let $q\colon W\times_{X}Y\ra Y$ be the projection onto the second coordinate and fix $(w_0,y_0)\in W\times_{X}Y$ so that $f(w_0)=x_0=p(y_0)$. Consider the pullback square\[\xymatrix{ \left(\mathcal{P}\left(W\times_{X}Y\right)\right)_{(w_0,y_0)} \ar@{-->}[dr]_{\cong}^-{\Psi} \ar@/^1pc/[drr]^-{\mathcal{P}q} \ar@/_1pc/[ddr]_-{\mathcal{P}(f^{\ast}p)} \\
& (\mathcal{P}W)_{w_0}\times_{\pxxo}\pyyo \ar[r]_-{r_1} \ar[d]_{\cong}^-{r_2} & \pyyo \ar[d]_{\cong}^-{\mathcal{P}p} \\
& (\mathcal{P}W)_{w_0} \ar[r]_-{\mathcal{P}f} & \pxxo }\] where $r_1,r_2$ are the projections. The maps $\mathcal{P}(f^{\ast}p)$ and $\mathcal{P}q$ induce the canonical homeomorphism $\Psi$.  Since $\mathcal{P}p$ is a homeomorphism by assumption, categorical considerations give that $r_2$ is a homeomorphism. Thus $\mathcal{P}(f^{\ast}p)$ is a homeomorphism. The same argument may be used to show that $\Phi(f^{\ast}p)\colon\left(\Phi \left(W\times_{X}Y\right)\right)_{(w_0,x_0)}\ra \left(\Phi W\right)_{w_0}$ is a homeomorphism. The fact that a morphism $g\colon Y\ra Y'$ of semicoverings $p$ and $p'$ of $X$ induces a map $f^{\ast}g\colon W\times_{X}Y\ra W\times_{X}Y'$ such that $f^{\ast}p'\circ f^{\ast}g=f^{\ast}p$ (i.e. a morphism $f^{\ast}p\ra f^{\ast}p'$) follows from the universal property of $W\times_{X}Y'$. Thus $f^{\ast}\colon\mathbf{SCov}(X)\ra \mathbf{SCov}(W)$ is a functor.
\end{proof}
\begin{proposition}
If $\pyx$ and $f\colon Z\ra X$ are connected semicoverings and $p$ is universal, the induced semicovering $f^{\ast}p$ of $Z$ is also universal.
\end{proposition}
\section{$\spaces$-groupoids and $\pi^{\tau} X$} \label{topgroupoidsection}
In this section, we provide a construction for a topology on the fundamental groupoid which plays an important role in the classification of semicoverings.

The \textit{path component space} $\pi_{0}^{qtop}Z$ of a space $Z$ is the set of path components $\pi_{0}Z$ viewed as a quotient space of $Z$. That $\pi_{0}^{qtop}$ gives endofunctors of $\spaces$ and $\bspaces$ follows directly from the universal property of quotient spaces. The group $\pi_{1}^{qtop}(X,x)=\pi_{0}^{qtop}\Omega(X,x)$ is the \textit{quasitopological fundamental group} of $(X,x)$ and is characterized by the canonical map $h\colon\Omega(X,x)\ra \pi_{1}^{qtop}(X,x)$ identifying homotopy classes of maps being quotient. This is a quasitopological group in the sense that inversion is continuous and left and right translations by fixed elements are continuous. It is known that $\pi_{1}^{qtop}$ is a homotopy invariant which takes values in the category $\qtg$ of quasitopological groups and continuous group homomorphisms \cite{CM}, however, $\pi_{1}^{qtop}(X,x)$ often fails to be a topological group \cite{Br10.1,Fab10,Fab11}.

In \cite{Br10.2}, this failure is repaired within $\qtg$ by noticing that the forgetful functor $\tg\to \qtg$ has a left adjoint $\tau$ such that the two triangles in \[\xymatrix{ \qtg \ar[dr] \ar@<.5ex>[rr]^-{\tau} && \tg \ar@<.5ex>[ll] \ar[dl]\\ & \grp }\] commute (the unlabeled arrows are forgetful functors). To construct $\tau$ explicitly, let $F_{M}(G)$ be the free (Markov) topological group on the underlying space of $G$ and $m_G\colon F_{M}(G)\ra G$ be the multiplication of letters induced by the identity of $G$. Let $\tau(G)$ be the underlying group of $G$ with the quotient topology with respect to $m_G$. Since $F_{M}(G)$ is a topological group, so is the quotient group $\tau(G)$. Applying $\tau$ has the effect of removing the smallest number of open sets from the topology of a quasitopological group $G$ so that one obtains a topological group. Thus $\tau(G)=G$ if and only if $G$ is already a topological group. It is then natural to define $\pi_{1}^{\tau}=\tau\circ \pi_{1}^{qtop}$.
\begin{proposition}
The topology of $\pi_{1}^{\tau}(X,x_0)$ is the finest group topology on $\pi_{1}(X,x_0)$ such that $h\colon\Omega(X,x_0)\to \pi_{1}(X,x_0)$ is continuous.
\end{proposition}
\begin{proof}
Suppose $\pi_{1}(X,x_0)$ is endowed with a topology making it a topological group and such that $\Omega(X,x_0)\to \pi_{1}(X,x_0)$ is continuous. The identity $\pi_{1}^{qtop}(X,x_0)\to \pi_{1}(X,x_0)$ is continuous by the universal property of quotient spaces and since $\pi_{1}(X,x_0)$ is a topological group, the adjoint is the continuous identity $\pi_{1}^{\tau}(X,x_0)=\tau(\pi_{1}^{qtop}(X,x_0))\to \pi_{1}(X,x_0)$. Thus the topology of $\pi_{1}^{\tau}(X,x_0)$ is finer than that of $\pi_{1}(X,x_0)$.
\end{proof}
This construction is now extended to the fundamental groupoid.
\begin{definition} \emph{
A $\mathbf{qTop}$-\textit{groupoid} is a (small) groupoid $\sG$ where the object set $X$ and hom-sets $\sG(x,y)$ are equipped with topologies such that each composition $\sG(x,y)\times \sG(y,z)\ra \sG(x,z)$ is continuous in each variable and each inversion function $\sG(x,y)\ra \sG(y,x)$ is continuous. A morphism of $\mathbf{qTop}$-groupoids is a functor $F\colon\sG\ra \sG '$ such that each function $F\colon\sG(x,y)\ra \sG '(F(x),F(y))$, $f\mapsto F(f)$ is continuous. If each composition map in $\sG$ is jointly continuous, then $\sG$ is a $\mathbf{Top}$-groupoid, that is, a groupoid enriched over $\spaces$ in the sense of \cite{KellyEnriched}. Let $\mathbf{qTopGrpd}$ be the category of $\mathbf{qTop}$-groupoids and $\mathbf{TopGrpd}$ be the full subcategory of $\mathbf{Top}$-groupoids.}
\end{definition}
Note that $\sG(x,y)\times \sG(y,z)\ra \sG(x,z)$ is continuous in each variable if and only if all translations $\lambda_{f}\colon\sG(x,y)\ra \sG(x,z)$, $g\mapsto fg$ and $\rho_{k}\colon\sG(y,z)\ra \sG(x,z)$, $g\mapsto gk$ are homeomorphisms. The author emphasizes that the notion of $\spaces$-groupoid is distinct from that of \textit{topological groupoid} which refers to a groupoid internal to $\spaces$.
\begin{proposition}
Let $\pi^{qtop} X$ denote the fundamental groupoid of $X$ where each hom-set $\pi^{qtop} X(x_1,x_2)$ is viewed as the quotient space $\pi_{0}^{qtop}\left(\mathcal{P}X(x_1,x_2)\right)$. This gives the fundamental groupoid the structure of a $\mathbf{qTop}$-groupoid. Moreover, $\pi^{qtop}\colon\spaces\ra \mathbf{qTopGrpd}$ is a functor.
\end{proposition}
\begin{proof}
By applying $\pi_{0}^{qtop}$ to operations of left concatenation $\alpha\to \alpha\ast \beta$, right concatenation $\alpha\to\beta\ast\alpha$, and inversion $\alpha\mapsto \alpha^{-1}$ on path spaces, one observes that $\pi^{qtop}X$ is indeed a $\mathbf{qTop}$-groupoid. Similarly, a map $f\colon X\ra Y$ induces the map $\pi_{0}^{top}(\mathcal{P}(f))\colon\pi^{qtop}X(x_1,x_2)\to \pi^{qtop}Y(f(x_1),f(x_2))$, $[\alpha]\mapsto [f\circ \alpha]$.
\end{proof}
The following lemma extends the definition of $\tau$ from groups to groupoids by applying the group-valued $\tau$ to vertex groups and extending via translations.
\begin{lemma}
The forgetful functor $\topgrpd\ra \qtopgrpd$ has a left adjoint $\tau\colon\qtopgrpd\ra \topgrpd$ which is the identity on underlying groupoids.
\end{lemma}
\begin{proof}
Let $\sG$ be a $\mathbf{qTop}$-groupoid. For each $x\in Ob(\sG)$, let $\tau(\sG)(x)$ be the topological group $\tau(\sG(x))$. If $x\neq y$ and $\sG(x,y)\neq\emptyset$, let $\tau(\sG)(x,y)$ have the topology generated by the sets $Ug=\{ug|u\in U\}$ where $g\in \sG(x,y)$ and $U$ is open in $\tau(\sG)(x)$. Since $Ug_{1}g_{2}^{-1}$ is open in $\tau(\sG)(x)$ for all $g_1,g_2\in \sG(x,y)$, the right translations $\rho_{g}\colon\tau(\sG)(x)\ra \tau(\sG)(x,y)$ are homeomorphisms. Note that if $g\in \sG(x,y)$, then $\lambda_{g^{-1}}\circ \rho_{g}\colon\sG(x)\ra \sG(y)$, $h\mapsto g^{-1}hg$ is an isomorphism of quasitopological groups. The functorality of $\tau\colon\qtg\ra \tg$ then gives that the same homomorphism $\tau(\sG)(x)\ra \tau(\sG)(y)$ is an isomorphism of topological groups. Thus all left translations $\lambda_{g}\colon\tau(\sG)(y)\ra \tau(\sG)(x,y)$ are homeomorphisms. One now may use the fact that the vertex groups $\tau(\sG)(x)$ are topological groups to see that $\tau(\sG)$ is a $\spaces$-groupoid. A morphism $F\colon\sG\ra \sG '$ of $\mathbf{qTop}$-groupoids induces a morphism $\tau(F)\colon\tau(\sG)\ra \tau(\sG ')$ of $\spaces$-groupoids since the group-valued $\tau$ on the vertex groups gives continuous homomorphisms $\tau(\sG)(x)\ra \tau(\sG ')(F(x))$. One may then extend to all hom-sets via translations. We use a similar argument to illustrate the universal property of $\tau(\sG)$. Suppose $\sG '$ is a $\spaces$-groupoid and $F\colon\sG\to \sG '$ is a morphism of $\mathbf{qTop}$-groupoids. It suffices to show that each function $F\colon\tau(\sG)(x,y)\ra \sG '(F(x),F(y))$ is continuous. Note that for each $x\in Ob(\sG)$, $F\colon\sG(x)\to \sG '(F(x))$ is a continuous group homomorphism from a quasitopological group to a topological group. The adjoint homomorphism $F\colon\tau(\sG)(x)=\tau(G(x))\to \sG '(F(x))$ is also continuous. Again, we extend via translations to find that $F\colon\tau(\sG)(x,y)\ra \sG '(F(x),F(y))$ is continuous in general.
\end{proof}
By construction, the vertex groups of $\tau(\sG)$ are the topological groups $\tau(\sG(x))$. Since each identity $\sG(x)\ra \tau(\sG)(x)$ is continuous, it follows that the identity functor $\sG\ra \tau(\sG)$ is a morphism of $\mathbf{qTop}$-groupoids. Just as in the group case a $\mathbf{qTop}$-groupoid $\sG$ is a $\mathbf{Top}$-groupoid if and only if $\sG=\tau(\sG)$.
\begin{definition} \emph{
The \textit{fundamental} $\spaces$\textit{-groupoid} of a topological space $X$ is the $\spaces$-groupoid $\pi^{\tau}X=\tau(\pi^{qtop} X)$.}
\end{definition}

For practicality, we introduce an alternative construction of $\pi^{\tau} X$. The following approximation technique, which is possible since the existence of $\pi^{\tau} X$ is already known, extends to groupoids the well-known process of inductively forming quotient topologies on groups. The group case of what is given here is laid out in more detail in \cite{Br10.2}.
\begin{approximation}\label{approximation} \emph{
Let $\sG=\sG_0$ be a $\mathbf{qTop}$-groupoid. Construct $\mathbf{qTop}$-groupoids $\sG_{\zeta}$ inductively so that if $\zeta$ is a successor ordinal, the topology of $\sG_{\zeta}(x,y)$ (provided it is non-empty) is the quotient topology with respect to the sum of multiplication maps \[\mu\colon\coprod_{a\in Ob(G)}\sG_{\zeta-1}(x,a)\times \sG_{\zeta-1}(a,y)\to \sG_{\zeta}(x,y).\]If $\zeta$ is a limit ordinal, the topology of $\sG_{\zeta}(x,y)$ is the intersection of the topologies of $\sG_{\eta}(x,y)$ for $\eta<\zeta$.}
\end{approximation}
\begin{lemma} Let $\sG$ be a $\mathbf{qTop}$-groupoid.
\begin{enumerate}
\item The identities $\sG_{\zeta}\ra \sG_{\zeta+1}\ra \tau(\sG)$ are morphisms of $\mathbf{qTop}$-groupoids for each $\zeta$.
\item $\sG_{\zeta}$ is a $\mathbf{qTop}$-groupoid for each $\zeta$.
\item $\tau(\sG_{\zeta})=\tau(\sG)$ for each $\zeta$.
\item $\sG_{\zeta}$ is a $\mathbf{Top}$-groupoid if and only if $\sG_{\zeta}=\tau(\sG)$ if and only if $\sG_{\zeta}(x,y)=\sG_{\zeta+1}(x,y)$ for all $x,y\in Ob(\sG)$.
\item There is an ordinal number $\zeta_{0}$ such that $\sG_{\zeta}=\tau(\sG)$ for each $\zeta\geq \zeta_{0}$.
\end{enumerate}
\end{lemma}
\begin{proof}

1. This follows from transfinite induction. The case for limit ordinals is clear. For successor ordinal $\zeta$, each map $\sG_{\zeta-1}(x,y)\ra \sG_{\zeta}(x,y)$ is continuous since $\sG_{\zeta-1}(x,y)\times \{id_{y}\}\subset\coprod_{a\in Ob(G)}\sG_{\zeta-1}(x,a)\times \sG_{\zeta-1}(a,y)$ and $\mu$ is continuous. Additionally, the left vertical map in the following diagram is quotient.
\[\xymatrix{
 \coprod_{a\in Ob(\sG)}\sG_{\zeta-1}(x,a)\times \sG_{\zeta-1}(a,y) \ar[r]^-{id} \ar[d]_-{\mu} & \coprod_{a\in Ob(\sG)}\tau(\sG)(x,a)\times\tau(\sG)(a,y) \ar[d]_-{\mu} \\ \sG_{\zeta}(x,y) \ar[r]_-{id} & \tau(\sG)(x,y) }\]Therefore if the top map is continuous, so is the bottom map.

2. The continuity of translations and inversion in $\sG_{\zeta}$ follows from a straightforward transfinite induction argument similar to that in 1.

3. Since $id\colon\sG_{\zeta}\to \tau(\sG)$ is a morphism of $\mathbf{qTop}$-groupoids so is $id\colon\tau(\sG_{\zeta})\to \tau(\tau(\sG))=\tau(\sG)$. Additionally $\sG\ra \sG_{\zeta}\to \tau(\sG_{\zeta})$ is a morphism of $\mathbf{qTop}$-groupoids whose adjoint is the inverse $id\colon\tau(\sG)\to \tau(\sG_{\zeta})$.

4. The first biconditional is clear from 3. The second is clear from the observation that $\sG_{\zeta}$ is a $\spaces$-groupoid if and only if $\mu\colon\coprod_{a\in Ob(\sG)}\sG_{\zeta}(x,a)\times \sG_{\zeta}(a,y)\to \sG_{\zeta}(x,y)$ is continuous for each $x,y\in Ob(\sG)$.

5. For each ordinal $\zeta$, let $A_{\zeta}=\coprod_{x,y\in Ob(\sG)}\sG_{\zeta}(x,y)$ be the disjoint union of hom-spaces and $\mathscr{T}_{\zeta}$ be the topology of $A_{\zeta}$. 1. gives that $\mathscr{T}_{\zeta+1}\subseteq \mathscr{T}_{\zeta}\subseteq \mathscr{T}_{0}$ for each $\zeta$ and 4. implies that $\mathscr{T}_{\zeta+1}= \mathscr{T}_{\zeta}$ if and only if $\sG_{\zeta}=\tau(\sG)$. Suppose $\sG_{\zeta}\neq\tau(\sG)$ for each $\zeta$. Thus $\mathscr{T}_{\zeta}-\mathscr{T}_{\zeta+1}\neq\emptyset$ for each ordinal $\zeta$, contradicting the fact there is no injection of ordinal numbers into the power set of $\mathscr{T}_{0}$. Thus there is an ordinal $\zeta_{0}$ such that $\sG_{\zeta_{0}}=\tau(\sG)$. Since $\sG_{\zeta_{0}}\ra \sG_{\zeta}\ra \tau(\sG)$ are morphisms of $\mathbf{qTop}$-groupoids whenever $\zeta\geq \zeta_{0}$, it follows that $\sG_{\zeta}=\tau(\sG)$ for all $\zeta\geq \zeta_{0}$.
\end{proof}
\begin{corollary}
For each $x_1,x_2\in X$, the canonical maps $h\colon\px(x_1,x_2)\ra \pi^{\tau}X(x_1,x_2)$ identifying homotopy classes of paths are continuous.
\end{corollary}
\begin{proof}
The topology of $\pi^{\tau}X(x_1,x_2)$ is coarser than that of $\pi^{qtop}X(x_1,x_2)$ and $h\colon\px(x_1,x_2)\ra \pi^{qtop}X(x_1,x_2)$ is continuous by definition.
\end{proof}
\section{Open embeddings and enriched monodromy}
\begin{lemma} \label{embeddingpathspaces}
If $\pyx$ is a semicovering such that $p(y_i)=x_i$, $i=1,2$, the map $\mathcal{P}p\colon\py(y_1,y_2)\ra \px(x_1,x_2)$ is an open embedding.
\end{lemma}
\begin{proof}
Since $\mathcal{P}p\colon\py(y_1,y_2)\ra \px(x_1,x_2)$ is the restriction of the homeomorphism $\mathcal{P}p\colon(\py)_{y_1}\ra (\px)_{x_1}$, it suffices to show the image of $\mathcal{P}p$ is open in $\px(x_1,x_2)$. Let $\alpha\in \px(x_1,x_2)$ such that $\tilde{\alpha}_{y_{1}}\in \py(y_1,y_2)$. Let $\mathcal{U}=\bigcap_{j=1}^{n}\langle K_{n}^{j},U_j\rangle$ be an open neighborhood of $\tilde{\alpha}_{y_{1}}$ in $(\py)_{y_1}$ such that $p|_{U_n}\colon U_n\ra p(U_n)$ is a homeomorphism. Since $\mathcal{P}p\colon(\py)_{y_1}\cong(\px)_{x_1}$, $\mathcal{W}=\mathcal{P}p(\mathcal{U})\cap \px(x_1,x_2)$ is an open neighborhood of $\alpha$ in $\px(x_1,x_2)$. If $\beta\in \mathcal{W}$, then $\mathcal{U}$ is an open neighborhood of $\tilde{\beta}_{y_{1}}$ in $(\py)_{y_1}$. In particular, $\tilde{\beta}_{y_{1}}(1)\in U_n\cap p^{-1}(x_2)=\{y_{2}\}$. Therefore $\tilde{\beta}_{y_{1}}\in \py(y_1,y_2)$ giving the inclusion $\mathcal{W}\subseteq Im\left(\mathcal{P}p\right)$.
\end{proof}
\begin{theorem} \label{embedding}
Let $\pyx$ be a semicovering map. The covering morphism $\pi^{\tau} p\colon\pi^{\tau} Y\ra \pi^{\tau} X$ is an open $\spaces$-functor.
\end{theorem}
\begin{proof}
First, note that each function $\pi p\colon\pi Y(y_1,y_2)\to \pi X(p(y_1),p(y_2))$ is injective since $\pi p$ is a covering morphism of groupoids. This injectivity is independent of topologies on the hom-sets. We take the inductive approach to $\pi^{\tau}$ discussed in \ref{approximation}. For simplicity of notation, let $\sH_{0}=\pi^{qtop}Y$ and $\sG_{0}=\pi^{qtop}X$ and inductively take $\sH_{\zeta}$ and $\sG_{\zeta}$ to be the approximating $\mathbf{qTop}$-groupoids of $\tau(\sH_{0})=\pi^{\tau}Y$ and $\tau(\sG_{0})=\pi^{\tau} X$ respectively. For the first inductive step, we show that whenever $p(y_i)=x_i$, $i=1,2$, the map $\pi p\colon \sH_{0}(y_1,y_2)\to \sG_{0}(x_1,x_2)$, $[\alpha]\mapsto [p\circ \alpha]$ is open.

Let $U$ be an open neighborhood in $\sH_{0}(y_1,y_2)$. The diagram \[\xymatrix{ \py(y_1,y_2) \ar[r]^{\mathcal{P}p} \ar[d]_{h_{Y}} & \px(x_1,x_2) \ar[d]^{h_X} \\ \sH_{0}(y_1,y_2) \ar[r]_{\pi p} & \sG_{0}(x_1,x_2) }\]commutes when $h_{Y}$ and $h_{X}$ are the canonical quotient maps. Since $h_X$ is quotient and the top map is open (Lemma \ref{embeddingpathspaces}), it suffices to show $h_{X}^{-1}(\pi p(U))=\mathcal{P}p(h_{Y}^{-1}(U))$. If $\alpha\in h_{X}^{-1}(\pi p(U))$, then $[\alpha]\in \pi p(U)$ and the lift $\tilde{\alpha}_{y_{1}}$ ends at $y_{2}$ by 3. of Lemma \ref{contliftingprops}. Now that $\pi p\left(\left[\tilde{\alpha}_{y_{1}}\right]\right)=[\alpha]\in \pi p(U)$, the injectivity of $\pi p$ gives $\left[\tilde{\alpha}_{y_{1}}\right]\in U$. Therefore $\alpha=p\circ \tilde{\alpha}_{y_{1}}=\mathcal{P}p(\tilde{\alpha}_{y_{1}})$ for $\tilde{\alpha}_{y_{1}}\in h_{Y}^{-1}(U)$. For the other inclusion, if $\alpha=p\circ \tilde{\alpha}_{y_{1}}$ such that $\left[\tilde{\alpha}_{y_{1}}\right]\in U$, then $[\alpha]=\pi p\left(\left[\tilde{\alpha}_{y_{1}}\right]\right)\in \pi p(U)$ and therefore $\alpha\in h_{X}^{-1}(\pi p(U))$.

Suppose $\zeta$ is an ordinal and that $\pi p\colon \sH_{\eta}(y_1,y_2)\to \sG_{\eta}(p(y_1),p(y_2))$ is open for each $y_1,y_2\in Y$ and each ordinal $\eta<\zeta$. Fix $y_1,y_2\in Y$ and let $p(y_i)=x_i$. Clearly, if $\zeta$ is a limit ordinal, then $\pi p\colon\sH_{\zeta}(y_{1},y_{2})\ra \sG_{\zeta}(x_1,x_2)$ is an open embedding. If $\zeta$ is a successor ordinal, consider the following diagram.
\[\xymatrix{\coprod_{b\in Y}\sH_{\zeta-1}(y_1,b)\times \sH_{\zeta-1}(b,y_2) \ar[d]_-{\mu_{Y}} \ar[rr]^-{\mathscr{P}_{\zeta-1}} && \coprod_{a\in X}\sG_{\zeta-1}(x_1,a)\times \sG_{\zeta-1}(a,x_2) \ar[d]^-{\mu_{X}}\\
\sH_{\zeta}(y_{1},y_{2}) \ar[rr]_-{\pi p} && \sG_{\zeta}(x_1,x_2) }\]The vertical multiplication maps are quotient by definition. The top map $\mathscr{P}_{\zeta-1}$ is, on each summand, the product of open embeddings (by induction hypothesis): \[\pi p\times \pi p\colon\sH_{\zeta-1}(y_{1},b)\times \sH_{\zeta-1}(b,y_{2})\ra \sG_{\zeta-1}(x_1,p(b))\times \sG_{\zeta-1}(p(b),x_2)\]where $([\alpha],[\beta])\mapsto ([p\circ \alpha],[p\circ \beta])$. Therefore $\mathscr{P}_{\zeta-1}$ is continuous and open. By the universal property of quotient spaces, the bottom map is continuous. Now suppose $U$ is open in $\sH_{\zeta}(y_{1},y_{2})$. It suffices to show $\mu_{X}^{-1}(\pi p(U))$ is open. If $([\delta],[\epsilon])\in \mu_{X}^{-1}(\pi p(U))$, then $[\delta\ast\epsilon]\in \pi p(U)$. Let $a_0=\delta(1)$. Consequently, the lift of $\delta\ast\epsilon$ starting at $y_{1}$ ends at $y_{2}$. This lift is $\tilde{\delta}_{y_{1}}\ast \tilde{\epsilon}_{b_0}$ where $b_0=\tilde{\delta}_{y_{1}}(1)\in p^{-1}(a_0)$. Since \[\pi p\left(\left[\tilde{\delta}_{y_{1}}\ast\tilde{\epsilon}_{b_0}\right]\right)=\left[p\circ \left(\tilde{\delta}_{y_{1}}\ast\tilde{\epsilon}_{b_0}\right)\right] = \left[\left(p\circ\tilde{\delta}_{y_{1}}\right)\ast\left(p\circ \tilde{\epsilon}_{b_0}\right)\right] = [\delta\ast\epsilon]\in \pi p(U)\]and $\pi p$ is injective, $\left[\tilde{\delta}_{y_{1}}\ast\tilde{\epsilon}_{b_0}\right]\in U$. Therefore $\mu_{Y}^{-1}(U)$ is an open neighborhood of $\left(\left[\tilde{\delta}_{y_{1}}\right],\left[\tilde{\epsilon}_{b_0}\right]\right)$ and $\mathscr{P}_{\zeta-1}(\mu_{Y}^{-1}(U))$ is an open neighborhood of $([\delta],[\epsilon])$ in $\coprod_{a\in X}\sG_{\zeta-1}(x_1,a)\times \sG_{\zeta-1}(a,x_2)$. It then suffices to check the inclusion $\mathscr{P}_{\zeta-1}(\mu_{Y}^{-1}(U)) \subseteq \mu_{X}^{-1}(\pi p(U))$. This follows easily from noticing that if $([\alpha],[\beta])\in \mu_{Y}^{-1}(U)$, then $[\alpha\ast \beta]\in U$ and \[\mu_{X}(\mathscr{P}_{\zeta-1}([\alpha],[\beta]))=[p\circ \alpha][p\circ \beta] = [p\circ (\alpha\ast \beta)] = \pi p([\alpha\ast \beta])\in \pi p(U)\]
\end{proof}
\begin{corollary}
If $\pyx$ is a semicovering such that $p(y_0)=x_0$, then $\pi_{1}^{\tau}p\colon\pi_{1}^{\tau}(Y,y_0)\to \pi_{1}^{\tau}(X,x_0)$ is an open embedding of topological groups.
\end{corollary}
The monodromy of a semicovering becomes an enriched functor when we use $\pi^{\tau} X$ and view $\set$ of non-empty sets as a $\spaces$-category by giving each set the discrete topology and endowing each hom-set $\set(F_1,F_2)$ with the topology of pointwise convergence (which is equivalent to the compact-open topology).
\begin{corollary}
The monodromy $\mathscr{M}p\colon\pi^{\tau} X\ra \set$ of a semicovering $\pyx$ is a $\spaces$-functor. Moreover, $\mathscr{M}\colon\scovx\ra \mathbf{TopFunc}(\pi^{\tau} X,\set)$ is a functor.
\end{corollary}
\begin{proof}
Suppose $p(y_i)=x_i$, $i=1,2$. Since the fibers of $p$ are discrete and $Im(\pi p\colon\pi Y(y_1,y_2)\to \pi X(x_1,x_2))=\{[\alpha]\in \pi X(x_1,x_2)|y_{1}\cdot [\alpha]=y_{2}\}$ is open in $\pi^{\tau} X(x_1,x_2)$ by \ref{embedding}, each  \[p^{-1}(x_1)\times \pi^{\tau} X(x_1,x_2)\ra p^{-1}(x_2)\text{ , }([\alpha],y)\mapsto y\cdot [\alpha]\]of the action of $\pi X$ on $Y$ is continuous. Since discrete spaces are locally compact Hausdorff, the adjoint \[\mathscr{M}p\colon\pi^{\tau} X(x_1,x_2)\ra \set(p^{-1}(x_1),p^{-1}(x_2))\text{ where }\mathscr{M}p([\alpha])(y)=y\cdot [\alpha]\]is continuous. Thus $\mathscr{M}p\colon\pi^{\tau} X\ra \set$ is a $\spaces$-functor. A morphism $f\colon Y\to Y'$ of coverings $\pyx$ and $p'\colon Y\ra X$ induces the natural transformation $\mathscr{M}f\colon\mathscr{M}p\ra \mathscr{M}p'$ with components $f\colon p^{-1}(x)\to (p ')^{-1}(x)$.
\end{proof}
\begin{corollary}
For each $x_0\in X$, monodromy of a semicovering $\pyx$ restricts to a continuous group homomorphism $\pitx \ra Homeo(p^{-1}(x_0))$.
\end{corollary}
Our main result is that $\mathscr{M}\colon\scovx\ra \mathbf{TopFunc}(\pi^{\tau} X,\set)$ is an isomorphism of categories for all $X$ in the class of (locally wep-connected) spaces introduced in the next section. The motivation for extending beyond locally path connected spaces to this class lies in potential applications to the theory of topological groups and $\spaces$-groupoids.
\section{Local properties of endpoints of paths} \label{localpropertiesofendpointsofpaths}
The fundamental group $\pi_{1}^{\tau}$ naturally realizes many important topological groups on spaces which are not locally path connected. We work to make our theory of semicoverings apply to such spaces. Since application to arbitrary spaces is unrealistic (recall Zeeman's example mentioned in the introduction), we must specify exactly which non-locally path connected spaces are to be included in our theory. The two notions of wep- and local wep-connectedness below generalize the notion of local path connectedness to suit this need. Definition \ref{localpathconn} appeared first in \cite{Br10.2} since it is precisely what is needed for the topological van Kampen theorem.
\begin{definition} \label{localpathconn}  \emph{Let $\alpha\colon I\ra X$ be a path.
\begin{enumerate}
\item $\alpha$ is \textit{\we} if for every open neighborhood $U$ of $\alpha$ in $\px$ there are open neighborhoods $V_0,V_1$ of $\alpha(0),\alpha(1)$ in $X$ respectively such that for every $a\in V_0, b\in V_1$ there is a path $\beta\in U$ with $\beta(0)=a$ and $\beta(1)=b$.
\item $\alpha$ is \textit{\wt} if for every open neighborhood $U$ of $\alpha$ in $(\px)_{\alpha(0)}$ there is an open neighborhood $V_1$ of $\alpha(1)$ such that for each $b\in V_1$, there is a path $\beta\in U$ with $\beta(1)=y$. 
\end{enumerate}
A space $X$ is \textit{{\stc}} if every pair of points in $X$ can be connected by a {\we} path.}
\end{definition}
Some intuition for {\we} and {\wt} paths is given in \cite{Br10.2}.
\begin{proposition} \label{quotienteval}
If $X$ is wep-connected and $x_0\in X$, the evaluation map $ev_1\colon\pxxo\ra X$, $\alpha\mapsto\alpha(1)$ is quotient.
\end{proposition}
\begin{proof}
Suppose $x\in U\subseteq X$ such that $ev_{1}^{-1}(U)$ is open in $\pxxo$. Since $X$ is \stc, there is a {\wt} path $\gamma\in \pxxo$ ending at $x$. Since $ev_{1}^{-1}(U)$ is an open neighborhood of $\gamma$, there is an open neighborhood $V$ of $x$ in $X$ such that for each $v\in V$ there is a path $\alpha\in ev_{1}^{-1}(U)$ from $x_0$ to $v$. Thus $V\subseteq U$.
\end{proof}
\begin{proposition}
If the path components of $X$ are wep-connected, then $\pi_{0}^{qtop}(X)$ is discrete, i.e. the path components of $X$ are open.
\end{proposition}
\begin{proof}
Let $x\in X$ and $\alpha$ be any well-ended path such that $\alpha(0)=x$. Since $\px$ is an open neighborhood of $\alpha$, there are open neighborhoods $V_0$, $V_1$ of $x$, $\alpha(1)$ respectively such that for each $a\in V_0$, $b\in V_1$ there is a path $\gamma\in \px$ from $a$ to $b$. Since for any $a\in V_0$, there is a path $\gamma$ from $a$ to $\alpha(1)$, it is clear that $V_0$ is contained in the path component of $x$.
\end{proof}
The previous two propositions indicate that spaces with wep-connected path components are suitable for application of Lemma \ref{generalliftinglemma}. Unfortunately, difficulties arise as one attempts to construct semicoverings of general wep-connected spaces. We are then motivated to slightly strengthen this notion.
\begin{definition} \label{localpathconn2} \emph{Let $\alpha\colon I\ra X$ be a path.
\begin{enumerate}
\item $\alpha$ is \textit{\ee} if for every open neighborhood $U$ of $\alpha$ in $\px$ there are open neighborhoods $V_0,V_1$ of $\alpha(0),\alpha(1)$ in $X$ respectively such that for every $a\in V_0, b\in V_1$ there is a well-ended path $\beta\in U$ with $\beta(0)=a$ and $\beta(1)=b$.
\item $\alpha$ is \textit{\et} if for every open neighborhood $U$ of $\alpha$ in $(\px)_{\alpha(0)}$ there is an open neighborhood $V_1$ of $\alpha(1)$ such that for each $b\in V_1$, there is a well-targeted path $\beta\in U$ with $\beta(1)=y$.
\end{enumerate}
A space $X$ is \textit{{\estc}} if every pair of points in $X$ can be connected by a {\ee} path.}
\end{definition}
\begin{remark}\emph{
The definitions of {\we} and {\wt} paths address the same property of the underlying space. It is shown in \cite{Br10.2} that for fixed $x\in X$, $X$ is {\stc} if and only if for each $x'\in X$, there is a {\wt} path from $x$ to $x'$. The analogous statement holds for {\estc} spaces and locally well-targeted paths. We repeatedly call upon this fact without reference. In situations where it is necessary to use a basepoint, it is more convenient to work with {\wt} and {\et} paths.}
\end{remark}
Clearly every {\estc} space is \stc. Examples of spaces which are {\stc} but not {\estc} exist but are complicated and would distract from our purposes. Additionally, if $\alpha\colon I\ra X$ is a path and $X$ is locally path connected at $\alpha(0)$ and $\alpha(1)$, then $\alpha$ is {\we} \cite[Prop. 6.5]{Br10.2}. It follows immediately that any path connected, locally path connected space is locally wep-connected. There are many non-locally path connected spaces which are locally wep-connected (for instance, see Prop. \ref{susisslwe} below). Since working with (locally) well-ended paths requires working with the compact-open topology, we make heavy use of the following machinery and notation (as is used in \cite{Br10.1,Br10.2}) for dealing with operating on neighborhoods of paths. 

Let $\mathcal{U}=\bigcap_{j=1}^{n}\langle C_j,U_j\rangle$ be an open neighborhood of a path $p\in \px$. Then $\mathcal{U}_{A}=\bigcap_{A\cap C_j\neq \emptyset}\langle T_{A}^{-1}(A\cap C_j),U_j\rangle$ is an open neighborhood of $p_{A}$. If $A=\{t\}$ is a singleton, then $\mathcal{U}_{A}=\bigcap_{t\in C_j}\langle I,U_j\rangle=\langle I,\bigcap_{t\in C_j}U_j\rangle$. On the other hand, if $p=q_{A}$ for some path $q\in \px$, then $\mathcal{U}^{A}=\bigcap_{j=1}^{n}\langle T_{A}(C_j),U_j\rangle$ is an open neighborhood of $q$. If $A=\{t\}$ so that $p_A=c_{p(t)}$, let $\mathcal{U}^{A}=\bigcap_{j=1}^{n}\langle \{t\},U_j\rangle$. The following observation illustrates how one may ``place one neighborhood after another" using this notation.
\begin{lemma} Let $\mathcal{U}=\bigcap_{j=1}^{n}\langle C_j,U_j\rangle$ be an open neighborhood in $\px$ such that $\bigcup_{j=1}^{n}C_j=I$. Then
\begin{enumerate}
\item For any closed interval $A\subseteq I$, $(\mathcal{U}^{A})_{A}=\mathcal{U}\subseteq (\mathcal{U}_{A})^{A}$
\item If $0=t_0\leq t_1\leq t_2\leq ...\leq t_n=1$, then $\mathcal{U}= \bigcap_{i=1}^{n}( \mathcal{U}_{[t_{i-1},t_i]})^{[t_{i-1},t_i]}.$
\end{enumerate}
\end{lemma}
For brevity, if $\mathcal{U}=\bigcap_{i=1}^{m}\langle C_{i},U_i\rangle$ and $\mathcal{V}=\bigcap_{j=1}^{n}\langle D_j,V_j\rangle$ are neighborhoods of $\alpha$ and $\beta$ respectively where $\alpha(1)=\beta(0)$, we write $\mathcal{U}\mathcal{V}$ for the neighborhood $\mathcal{U}^{\left[0,\frac{1}{2}\right]}\cap \mathcal{V}^{\left[\frac{1}{2},1\right]}$ of $\alpha\ast\beta$.

The author refers to the space $\sus=\frac{X\times I}{X\times \{0,1\}}$ of the next proposition (studied in detail in \cite{Br10.1}) as the \textit{generalized wedge of circles} on the topological space $X$ due to the fact that $\pi_{1}^{\tau}(\sus,x_0)$ is naturally isomorphic to the free topological group $F_{M}(\pi_{0}^{qtop}(X))$. It is desirable that generalized wedges be locally wep-connected for application of semicoverings to these groups.
\begin{proposition} \label{susisslwe}
For every space $X$, the generalized wedge of circles $\sus$ is locally wep-connected. Moreover, a space obtained by attaching n-cells, $n\geq 2$ to $\sus$ is locally wep-connected.
\end{proposition}
\begin{proof}
Since $\sus$ is locally path connected at its basepoint $x_0$, pick $x\wedge t\in X\wedge (0,1)$. Let $\alpha$ be any path from $x_0$ to $x \wedge t$ satisfying $\alpha(s)=x\wedge r(s)$, and $\alpha^{-1}(x_0)=\{0\}$. A basic open neighborhood of $\alpha$ may be taken to be of the form \[\mathcal{U}=\left\langle K_{n}^{1},V\right\rangle\cap\bigcap_{j=2}^{n}\langle K_{n}^{j},U\wedge (a_j,b_j)\rangle\] where $V$ is a basic neighborhood of $x_0$ and $U$ is an open neighborhood of $x$ in $X$. If $u\wedge v\in U\wedge (a_n,b_n)$, then the path $\beta(s)=u\wedge r(s)$ from $x_0$ to $u\wedge t$ lies in $\mathcal{U}$ and also satisfies $\beta^{-1}(x_0)=\{0\}$. Let $\gamma$ be the arc $\gamma(s)=u\wedge(t+s(v-t))$ in $\{u\}\wedge (a_n,b_n)$ from $u\wedge t$ to $u\wedge v$ and $\delta$ be the path satisfying $\delta_{\left[0,\frac{2n-1}{2n}\right]}=\beta$ and $\delta_{K_{2n}^{2n}}=\gamma$. Clearly $\delta\in \mathcal{U}$ is a path from $x_0$ to $u\wedge v$ and therefore $\alpha$ is \wt. Since $\delta$ is also of the form $\delta(s)=y\wedge r'(s)$ where $\delta^{-1}(x_0)=\{0\}$, it follows that $\delta$ is \wt. Thus $\alpha$ is \et. Since n-cells are locally path connected, a straightforward extension of this argument applies to spaces obtained by attaching cells to $\sus$.
\end{proof}
\begin{corollary}
For any space $(Y,y_0)$, there is a locally wep-connected space $X$ and a map $X\to Y$ which induces an isomorphism $\pi_{1}^{\tau}(X,x_0)\ra \pi_{1}^{\tau}(Y,y_0)$ of topological groups.
\end{corollary}
\begin{proof}
The counit $cu\colon\Sigma\left(\Omega(Y,y_0)_+\right)\ra Y$ of the loop-suspension adjunction induces a quotient map \[\pi_{1}^{\tau}(cu)\colon\pi_{1}^{\tau}\left(\Sigma\left(\Omega(Y,y_0)_+\right),x_0\right)\ra \pi_{1}^{\tau}(Y,y_0)\] of topological groups. For each $[\beta]\in \ker \left(\pi_{1}^{\tau}(cu)\right)$, attach a 2-cell using a representative loop $\beta\colon S^1\ra \Sigma\left(\Omega(Y,y_0)_+\right)$.
Let $X$ be the resulting space, which is locally wep-connected by the previous proposition. The inclusion $j\colon\Sigma\left(\Omega(Y,y_0)_+\right)\hookrightarrow X$ induces a quotient map $\pi_{1}^{\tau}(j)\colon\pi_{1}^{\tau}\left(\Sigma\left(\Omega(Y,y_0)_+\right),x_0\right)\ra \pi_{1}^{\tau}(X,x_0)$ of topological groups by Corollary 5.3 of \cite{Br10.2}. Since $\ker\left(\pi_{1}^{\tau}(j)\right)=\ker\left(\pi_{1}^{\tau}(cu)\right)$, there is a map $X\ra Y$ which induces the desired isomorphism
\end{proof}
The following lemma requires simple arguments for the compact-open topology of path spaces. The statements with ``locally" included follow directly from the analogous statements which do not.
\begin{lemma} \label{lemmalwelwt} Let $X$ be a space and $\alpha\colon I\ra X$ be a path.
\begin{enumerate}
\item If there is a $0\leq t\leq 1$ such that $\alpha_{[t,1]}$ is {\wt} (resp. \et), then $\alpha$ is {\wt} (resp. \et).
\item If there are $0\leq s\leq t\leq 1$ such that $\alpha_{[t,1]}$ and $\left(\alpha_{[0,s]}\right)^{-1}$ are {\wt} (resp. \et), then $\alpha$ is {\we} (resp. \ee).
\item The reverse of a {\we} (resp. {\ee}) path is {\we} (resp. \ee).
\item The concatenation of {\we} (resp. {\ee}) paths is {\we} (resp. \ee).
\end{enumerate}
\end{lemma}
\begin{proof}
1. If $\alpha_{[t,1]}$ is {\wt} and $\mathcal{U}$ is an open neighborhood of $\alpha$, we find a neighborhood $\mathcal{V}$ of $\alpha$ of the form $\mathcal{V}=\bigcap_{j=1}^{n}\langle K_{n}^{j},U_j\rangle$ contained in $\mathcal{U}$. Now $\mathcal{V}_{[t,1]}$ is an open neighborhood of $\alpha_{[t,1]}$. Note that if $t=1$, $\alpha_{[t,1]}$ is constant at $\alpha(1)$, we may take $\mathcal{V}_{[t,1]}=\langle I,U_n\rangle$. By assumption, there is an open neighborhood $V$ of $\alpha(1)$ contained in $U_n$ such that for each $v\in V$ there is a path $\gamma\in \mathcal{V}_{[t,1]}$ from $\alpha(t)$ to $v$. When $t\neq 1$, the path $\beta$ satisfying $\beta_{[0,t]}=\alpha_{[0,t]}$ and $\beta_{[t,1]}=\gamma$ is the desired path in $\mathcal{V}$ (and thus in $\mathcal{U}$) from $\alpha(0)$ to $v$. When $t=1$, let $\beta$ be the path satisfying $\beta_{\left[0,\frac{n-1}{n}\right]}=\alpha_{\left[0,\frac{n-1}{n}\right]}$, $\beta_{K_{2n}^{2n-1}}=\alpha_{K_{n}^{n}}$, and $\beta_{K_{2n}^{2n}}=\gamma$. Now $\beta$ is a path in $\mathcal{V}$ from $\alpha(0)$ to $v$. In the case that $\alpha_{[t,1]}$ is {\et} and $t\neq 1$, we take $\gamma$ in the previous argument to be \wt. Since $\beta\in \mathcal{U}$ is such that $\beta_{[t,1]}=\gamma$ is \wt, $\beta$ itself is {\wt} and $\alpha$ is \et. Similarly, when $t\neq 1$, we have $s=\frac{2n-1}{2n}<1$ such that $\beta_{[0,s]}=\gamma$ is \wt. Thus $\beta$ is {\wt} and $\alpha$ is \et. 2. follows from the same type of argument used in 1. and 3. follows from the fact that $\delta^{-1}\in\mathcal{V}$ if and only if $\delta\in \mathcal{V}^{-1}$. 4. follows directly from 2.
\end{proof}
The following Corollary \ref{lwtclasses} allows us to replace any path in a {\estc} space by a homotopic (rel. endpoints) {\et} path.
\begin{corollary} \label{lwtclasses}
Let $X$ be {\stc} (resp. \estc) and $x_1,x_2\in X$. For each class $[\alpha]\in \pi X(x_1,x_2)$, there is a {\wt} (resp. \et) path $\beta\in [\alpha]$.
\end{corollary}
\begin{proof}
If $X$ is {\stc} (resp. \estc), there is a {\wt} (resp. \et) path $\gamma$ from $x_1$ to $x_2$. Let $\beta=\alpha\ast\gamma^{-1}\ast\gamma$. Clearly $\alpha\simeq \beta$ and 1. of Lemma \ref{lemmalwelwt} implies that $\beta$ is {\wt} (resp. \et). 
\end{proof}
The next two statements are motivated by the desire to lift properties of a space to its semicoverings.
\begin{proposition} \label{lwegeneralizedcoverings}
Let $\pyx$ be a semicovering map such that $p(y_0)=x_0$. If $\alpha\in \pxxo$ is (locally) well-targeted, then so is the lift $\tilde{\alpha}_{y_0}$.
\end{proposition}
\begin{proof}
Suppose $\alpha$ is {\wt} and let $\mathcal{W}$ be an open neighborhood of $\tilde{\alpha}_{y_0}$ in $\pyyo$. Since $p$ is a local homeomorphism, there is an open neighborhood $U$ of $\tilde{\alpha}_{y_0}(1)$ mapped homeomorphically by $p$ onto an open subset of $X$. Let $\mathcal{U}=\mathcal{W}\cap \langle \{1\},U\rangle$. Since $\mathcal{P}p\colon\pyyo\ra \pxxo$ is a homeomorphism, $\mathcal{V}=\mathcal{P}p\left(\mathcal{U}\right)$ is an open neighborhood of $\alpha$. By assumption, there is an open neighborhood $V$ of $\alpha(1)$ (which we may take to be contained in $p(U)$) such that for each $v\in V$ there is a path $\gamma\in \mathcal{V}$ from $x_0$ to $v$. Now $W=p^{-1}(V)\cap U$ is a homeomorphic copy of $V$ in $U$. If $w\in W$, then $p(w)\in V$ and there is a path $\gamma\in \mathcal{V}$ from $x_0$ to $p(w)$. Since $L_p\colon\mathcal{V}\cong \mathcal{U}$, the lift $\tilde{\gamma}_{y_0}$ of $\gamma$ lies in $\mathcal{U}$. Since $p\circ \tilde{\gamma}_{y_0}(1)=p(w)$ and $\tilde{\gamma}_{y_0}(1)\in p^{-1}(p(w))\cap U=\{w\}$, we have $\tilde{\gamma}_{y_0}(1)=w$. Since we have already shown that lifts of {\wt} paths are {\wt}, the {\et} case follows from the same argument and taking $\gamma$ to be \wt.
\end{proof}
\begin{corollary} \label{liftingproperties}
Let $\pyx$ be a semicovering map. If $X$ is locally path connected, then so is $Y$. If $X$ is {\stc} (resp. \estc), then so is every path component of $Y$.
\end{corollary}
\begin{proof}
The locally path connected case is clear since $p$ is a local homeomorphism. For the other two cases, let $p(y_0)=x_0$ and show the path component of $y_0$ in $Y$ is {\stc} (resp. \estc). Suppose $y\in Y$, $p(y)=x$, and $\tilde{\gamma}_{y_0}$ be a path from $y_0$ to $y$ so that $\gamma=p\circ \tilde{\gamma}_{y_0}$ is a path from $x_0$ to $x$. By Corollary \ref{lwtclasses}, there is a {\wt} (resp. \et) path $\alpha$ from $x_0$ to $x$ homotopic to $\gamma$. This homotopy lifts to a homotopy of paths $\tilde{\gamma}_{y_0}\simeq \tilde{\alpha}_{y_0}$. In particular, $\tilde{\alpha}_{y_0}(1)=\tilde{\gamma}_{y_0}(1)=y$ and $\tilde{\alpha}_{y_0}$ is {\wt} (resp. \et) by Proposition \ref{lwegeneralizedcoverings}.
\end{proof}
\section{Classification Theorems}
Our main classification of semicoverings is the following theorem.
\begin{theorem} \label{classificationI}
Let $X$ be a locally wep-connected space. Monodromy \[\mathscr{M}\colon\scovx\ra \mathbf{TopFunc}(\pi^{\tau}X,\set)\] is an isomorphism of categories.
\end{theorem}
The main difficulty in the proof of this theorem is the existence of a semicovering whose monodromy is a given $\spaces$-functor $F\colon\pi^{\tau}X\ra \set$. This is the content of section 5.1. Section 5.2 completes the proof of Theorem \ref{classificationI} and sections 5.3 and 5.4 offer alternative classifications in terms of open covering morphisms and continuous actions of topological groups on discrete sets. Under the usual conditions, this classification reduces to the well-known classification of covering spaces.
\begin{corollary}
If $X$ is path connected, locally path connected, and semilocally 1-connected, then $\covx=\scovx$ and $\ccovx=\cscovx$.
\end{corollary}
\begin{proof}
For any such $X$ and $x_0\in X$, $\pi^{\tau}(X,x_0)$ is discrete \cite{Br10.2}. Thus $\pi^{\tau}X$ is a discrete groupoid. This gives the middle equality in: \[\scovx\cong\mathbf{TopFunc}(\pi^{\tau}X,\set)=\mathbf{Func}(\pi X,\set)\cong \covx.\]Any semicovering of $X$ equivalent to a covering of $X$ must itself be a covering.
\end{proof}
It remains to be seen whether or not there are more general conditions giving $\covx=\scovx$.
\subsection{Existence of semicoverings}
Let $X$ be a path connected space and $F\colon\pi X\ra \set$ be any functor such that each $F(x)$ is non-empty. Let $\txf=\bigcup_{x\in X}F(x)$ and $\pf\colon\txf\ra X$ be the surjection given by $\pf(F(x))=x$. Note that $F$ determines the right action of $\pi X$ on $\txf$ given by $y\cdot [\alpha]=F([\alpha])(y)\in F(\alpha(1))$ for $\alpha\in \pxx$ and $y\in F(x)$. Let \[\tf\colon\coprod_{x\in X}F(x)\times (\px)_{x} \ra \txf \] be the function given by $\tf(y,\alpha)=y\cdot [\alpha]$ for $(y,\alpha)\in F(x)\times  (\px)_{x}$. We view each fiber $F(x)$ as a discrete space and give $\txf$ the quotient topology with respect to $\tf$. Consider the diagram \[\xymatrix{ \coprod_{x\in X}F(x)\times (\px)_{x} \ar[dr]_-{ev} \ar[r]^-{\tf} & \txf \ar@{-->}[d]^-{\pf}\\ & X }\] where $ev$ is the continuous evaluation $ev(y,\alpha)=\alpha(1)$. Since $\alpha(1)=\beta(1)$ whenever $y\in F(\alpha(0))$ and $z\in F(\beta(0))$ and $y\cdot [\alpha]=z\cdot [\beta]$, $\pf$ is continuous by the universal property of quotient spaces. Since the topology of $\txf$ is characterized by the quotient map $\tf$, we sometimes write a generic element of $\txf$ as $\tf(y,\alpha)=y\cdot[\alpha]$.
\begin{remark} \label{thetadata}\emph{
Given $x_1,x_2\in X$ and $y_i\in F(x_i)$, let $\mathcal{H}(y_1,y_2)=\{[\alpha]\in \pi X(x_1,x_2)|y_1\cdot[\alpha]=y_2\}$. Of course, if $y_1=y=y_2$ and $x=p(y)$, $\mathcal{H}(y)=\mathcal{H}(y,y)$ is the stabilizer subgroup of $\pi_{1}(X,x)$ at $y$. Note that $y\cdot [\alpha]=y\cdot [\beta]$ if and only if $[\alpha\ast\beta^{-1}]\in \mathcal{H}(y)$. Additionally, we note that by Corollary \ref{lwtclasses}, if $X$ is (locally) wep-connected, then for each $y\cdot [\alpha]\in \txf$ there is a (locally) well-targeted path $\beta\in \px(\alpha(0),\alpha(1))$ such that $y\cdot [\alpha]=y\cdot [\beta]$.}
\end{remark}
\begin{proposition} \label{phopen}
If $X$ is \stc, then $\pf\colon\txf\ra X$ is open.
\end{proposition}
\begin{proof}
Let $W$ be open in $\txf$, $x_1\in \pf(W)$, and pick any $(y,\alpha)$ such that $y\cdot [\alpha]\in W$ and $\alpha(1)=x_1$. Let $x_0=\alpha(0)$ so that $y\in F(x_0)$. By Corollary \ref{lwtclasses}, there is a well-targeted path $\beta$ homotopic to $\alpha$ rel. endpoints. By Remark \ref{thetadata},  it is clear that $y\cdot [\beta]=y\cdot [\alpha]\in W$. Now $\{y\}\times \mathcal{U}=\tf^{-1}(W)\cap \left(\{y\}\times (\px)_{x_0}\right)$ for some open neighborhood $\mathcal{U}$ of $\beta$ in $(\px)_{x_0}$. Since $\beta$ is well-targeted, there is an open neighborhood $V$ of $x_1$ in $X$ such that for every $v\in V$, there is a path $\delta\in \mathcal{U}$ such that $\delta(1)=v$. Thus $y\cdot [\delta]\in W$ and $\pf(y\cdot [\delta])=\delta(1)=v$. This gives the inclusion $V\subseteq \pf(W)$.
\end{proof}
\begin{proposition} \label{morphisms}
If $F,G\colon\pi^{\tau}X\ra \set$ are $\spaces$-functors and $\eta\colon F\ra G$ is a $\spaces$-natural transformation, then there is a map $p_{F,G}\colon\txf\ra \tilde{X}_{G}$ such that the triangle \[\xymatrix{ \txf \ar[dr]_{\pf} \ar[rr]^{p_{F,G}} && \tilde{X}_{G} \ar[dl]^{p_G} \\ & X }\] commutes.
\end{proposition}
\begin{proof}
The top horizontal map in the diagram \[\xymatrix{ \coprod_{x\in X}F(x)\times (\px)_{x} \ar[rr]^-{\coprod_{x\in X}\eta_{x}\times id} \ar[d]_{\tf} && \coprod_{x\in X}G(x)\times (\px)_{x} \ar[d]^{\Theta_{G}} \\ \txf \ar[dr]_{\pf} \ar@{-->}[rr]^{p_{F,G}} && \tilde{X}_{G} \ar[dl]^{p_G} \\ & X }\] given by the components of $\eta$ is continuous since each $F(x)$ is discrete. We claim $p_{F,G}$ is given by $p_{F,G}(\Theta_{F}(y,\alpha))=\Theta_{G}(\eta_{\alpha(0)}(y),\alpha)$ to make the triangle commute. To check that $p_{F,G}$ is well-defined, it suffices to show that $G([\alpha])(\eta_{\alpha(0)}(y))=G([\beta])(\eta_{\beta(0)}(z))$ whenever $F([\alpha])(y)=F([\beta])(z)$. However, if $F([\alpha])(y)=F([\beta])(z)$, then $\alpha(1)=\beta(1)$ and the naturality of $\eta$ gives that \[G([\alpha])(\eta_{\alpha(0)}(y))=\eta_{\alpha(1)}(F([\alpha])(y))=\eta_{\beta(1)}(F([\beta])(z))=G([\beta])(\eta_{\beta(0)}(z)).\] Since $\tf$ is quotient, $p_{F,G}$ is continuous. The last statement follows directly from Proposition \ref{twoofthree}.
\end{proof}
%
%If, for a given space $X$, $\pf\colon\txf\ra X$ is a semicovering for each $\spaces$-functor $F$, then we construct the inverse $\mathscr{S}\colon\mathbf{TopFunc}(\pi^{\tau}X,\set)\ra \scovx$ of Monodromy by taking $\mathscr{S}(F)=p_F$ and $\mathscr{S}(\eta\colon F\ra G)=p_{F,G}$.
%
\begin{pathlifting} \label{descriptionoflifts}\emph{
Without any assumptions on $X$ or $F$, we find canonical lifts of paths with respect to $p_F$. For $y\in F(x)$ we construct a canonical section to \[\mathcal{P}p_F\colon\left(\mathcal{P}\txf\right)_{y}\ra (\mathcal{P}X)_{x}\]as a composition of continuous functions. Since multiplication $\mu\colon I\times I\ra I$ of real numbers is continuous, $\mu^{\#}\colon(\px)_{x}\to \spaces((I\times I,\{0\}\times I\cup I\times\{0\}),(X,\{x\}))$, $\beta\mapsto \beta\circ \mu$ on the relative mapping space is continuous. Additionally, \[r\colon\spaces((I\times I,\{0\}\times I\cup I\times\{0\}),(X,\{x\}))\ra \left(\mathcal{P}(\px)_{x}\right)_{c_{x}} \text{ , }r(\phi)(s)(t)=\phi(s,t)\] is a homeomorphism. Note that $r(\beta\circ \mu)(s)(t)=\beta(st)$ and therefore $r(\beta\circ \mu)(s)=\beta_{[0,s]}$. Lastly, the map $\mathcal{P}\tf\colon\left(\mathcal{P}\pxx \right)_{c_{x}}\ra \left(\mathcal{P}\txf\right)_{c_{y}}$ is obtained by applying $\mathcal{P}$ to the restriction of $\tf$ to $\{y\}\times \pxx$. Now let $L_F\colon(\px)_{x}\ra\left(\mathcal{P}\txf\right)_{y}$ be the composition $\mathcal{P}\tf\circ r \circ \mu^{\#}$ which takes $\beta$ to the path $\tilde{\beta}_{y}(s)=y\cdot\left[\beta_{[0,s]}\right]$. Since $p_F\left(y\cdot\left[\beta_{[0,s]}\right]\right)=\beta(s)$, $\tilde{\beta}_{y}$ is a lift of $\beta$ starting at $y$. Therefore, $L_F$ is the desired section.}
\end{pathlifting}
\begin{homotopylifting} \label{descriptionofhomotopylifts}\emph{
We take a similar approach to find canonical lifts of homotopies of paths by constructing a section of\[\Phi\pf\colon\left(\Phi\txf\right)_{y}\to (\Phi X)_{x}.\] for $y\in F(x)$. Since multiplication $m\colon I\times \Delta_2\ra \Delta_2$, $m(s,t,u)=(st,u)$ is continuous, $m^{\#}\colon(\Phi X)_{x}\to \spaces((I\times \Delta_2,I\times e_1\cup \Delta_2\times \{0\}),(X,x))$, $m^{\#}(\phi)(s,t,u)=\phi(st,u)$ is continuous. Additionally, the map \[r\colon\spaces((I\times \Delta_2,I\times e_1\cup \Delta_2\times \{0\}),(X,x))\ra \left(\Phi(\mathcal{P}X)_{x}\right)_{c_{x}}\] given by $r(K)(t,u)(s)=K(s,t,u)$ is a homeomorphism. Note that $\left(r(\phi\circ m)(x,y)\right)(s)=\phi(sx,y)$. Additionally, $\Phi\tf\colon\left(\Phi\pxx \right)_{c_{x}}\ra \left(\Phi\txf\right)_{y}$ is obtained by applying $\Phi$ to the restriction of $\tf$ to $\{y\}\times \pxx$ and $hL_F\colon(\Phi X)_{x}\ra \left(\Phi\txf\right)_{y}$ is the composition $\Phi\tf\circ  r\circ m^{\#}$. To see that $hL_F$ is a section of $\Phi\pf$ we check that $\pf(hL_F(\phi)(t,u))=\phi(t,u)$ for $(t,u)\in \Delta_2$. This is straightforward from the equation \[\pf(\tf(r(\phi\circ m)(t,u)))= \left(r(\phi\circ m)(t,u)\right)(1)=\phi(1t,u)=\phi(t,u)\]}
\end{homotopylifting}
\begin{theorem} \label{contliftingofpf}
The map $\pf\colon\txf\ra X$ has continuous lifting of paths and homotopies if and only if $\pf$ has unique path lifting.
\end{theorem}
\begin{proof}
According to \ref{descriptionoflifts} and \ref{descriptionofhomotopylifts}, for any $y\in F(x)$ both \[\mathcal{P}p_F\colon\left(\mathcal{P}\txf\right)_{y}\ra (\mathcal{P}X)_{x}\text{ and }\Phi p_F\colon\left(\Phi \txf\right)_{y}\ra (\Phi X)_{x}\] are topological retractions. Unique path lifting implies that both of these maps are injective and therefore homeomorphisms.
\end{proof}
From now on, suppose $X$ is locally wep-connected and $F\colon\pi^{\tau}X\ra \set$ is a $\spaces$-functor. We require these assumptions to obtain a simple basis for the topology of $\txf$. Since each map $\pi^{\tau}X(x_1,x_2)\ra \set(F(x_1),F(x_2))$ is continuous so is each adjoint action map $F(x_1)\times \pi^{\tau}X(x_1,x_2)\ra F(x_2)$, $(y,[\alpha])\mapsto y\cdot [\alpha]$. Thus if $y_i\in F(x_i)$, the set $\mathcal{H}(y_1,y_2)$ is open in $\pi^{\tau}X(x_1,x_2)$. Moreover, since $h\colon\px(x_1,x_2)\ra \pi^{\tau}X(x_1,x_2)$, $\alpha\mapsto [\alpha]$ is continuous, the pre-image $h^{-1}\left(\mathcal{H}(y_1,y_2)\right)$ is open in $\px(x_1,x_2)$.
\begin{basis}\emph{
Let $\alpha\in \pxxo$, $y_0\in F(x_0)$ and $U$ be an open neighborhood of $y_0\cdot[\alpha]$ in $\txf$. By Remark \ref{thetadata}, we may assume that $\alpha$ is locally well-targeted. Since $\alpha\ast\alpha^{-1}$ is a null-homotopic loop, $h^{-1}(\mathcal{H}(y_0))$ is an open neighborhood of $\alpha\ast\alpha^{-1}$ in $\Omega(X,x_0)$. Now find a neighborhood $\mathcal{U}=\bigcap_{i=1}^{m}\langle K_{m}^{i},A_i\rangle$ of $\alpha$ in $\pxxo$ such that
\begin{enumerate}
\item $\{y_0\}\times \mathcal{U}\subseteq \tf^{-1}(U)$
\item $\alpha\ast\alpha^{-1}\in\mathcal{U}\mathcal{U}^{-1}\cap \Omega(X,x_0)\subseteq h^{-1}(\mathcal{H}(y_0))$.
\end{enumerate}
Since $\alpha$ is locally well-targeted, there is an open neighborhood $V$ of $\alpha(1)$ contained in $A_m$ such that for each $v\in V$, there is a well-targeted path $\delta\in \mathcal{U}$ from $x$ to $v$. Let \[B(y_0\cdot[\alpha],\mathcal{U},V)=\tf\left(\{y_0\}\times \left(\mathcal{U}\cap \langle \{1\},V\rangle\right)\right).\]}
\end{basis}
\begin{lemma} \label{basisoftxf}
The sets $B(y_0\cdot[\alpha],\mathcal{U},V)$ form a basis for the topology of $\txf$. Moreover, $B(y_0\cdot[\alpha],\mathcal{U},V)$ is mapped homeomorphically onto $V$ by $\pf$.
\end{lemma}
\begin{proof}
Since $U$ is arbitrary and $B(y_0\cdot[\alpha],\mathcal{U},V)\subseteq U$, it suffices to show that $B(y_0\cdot[\alpha],\mathcal{U},V)$ is open in $\txf$. Since $\tf$ is quotient, we check that $\tf^{-1}\left(B(y_0\cdot[\alpha],\mathcal{U},V)\right)$ is open in $\coprod_{x\in X}F(x)\times (\px)_{x}$. If \[(y_1,\beta)\in \tf^{-1}\left(B(y_0\cdot[\alpha],\mathcal{U},V)\right)\cap \left(\{y_1\}\times (\px)_{x_1}\right),\] then $y_1\cdot[\beta]=y_0\cdot[\epsilon]$ (Recall that this implies $\beta(1)=\epsilon(1)$) for $\epsilon\in \mathcal{U}\cap \langle \{1\},V\rangle$. By assumption, there is a well-targeted path $\delta\in \mathcal{U}$ such that $\delta(1)=\epsilon(1)$. Since $\delta\ast\epsilon^{-1}\in\mathcal{U}\mathcal{U}^{-1}\cap \Omega(X,x_0)\subseteq h^{-1}(\mathcal{H}(y_0))$, we have $y_0\cdot[\delta]=y_0\cdot[\epsilon]=y_1\cdot[\beta]$. Thus $[\delta\ast\beta^{-1}]\in \mathcal{H}(y_0,y_1)$ and $h^{-1}(\mathcal{H}(y_0,y_1))$ is an open neighborhood of $\delta\ast\beta^{-1}$ in $\px(x_0,x_1)$. This observation guarantees that there are open neighborhoods $\mathcal{B}=\bigcap_{j=1}^{n}\langle K_{n}^{j},B_j\rangle$ of $\beta$ in $(\px)_{x_1}$ and $\mathcal{D}=\bigcap_{k=1}^{p}\langle K_{p}^{k},D_k\rangle$ of $\delta$ in $(\px)_{x_0}$ such that
\begin{enumerate}
\item $\mathcal{D}\subseteq \mathcal{U}$
\item $\delta\ast\beta^{-1}\in \mathcal{D}\mathcal{B}^{-1}\cap \px(x_0,x_1)\subseteq h^{-1}(\mathcal{H}(y_0,y_1))$
\item $B_n\cup D_p\subseteq V$.
\end{enumerate}
Since $\delta$ is well-targeted, there is an open neighborhood $W$ of $\delta(1)=\beta(1)$ in $B_n\cap D_p$ such that for each $w\in W$, there is a path $\zeta\in \mathcal{D}$ from $x_0$ to $w$. We claim the neighborhood $\{y_1\}\times \left(\mathcal{B}\cap \langle \{1\},W\rangle\right)$ of $(y_1,\beta)$ is contained in $\tf^{-1}\left(B(y_0\cdot[\alpha],\mathcal{U},V)\right)\cap \left(\{y_1\}\times (\px)_{x_1}\right)$. If $\gamma\in \mathcal{B}\cap \langle \{1\},W\rangle$, there is a path $\zeta\in \mathcal{D}$ from $x_0$ to $\gamma(1)$. This gives $\zeta\ast\gamma^{-1}\in \mathcal{D}\mathcal{B}^{-1}\cap \px(x_0,x_1)\subseteq h^{-1}(\mathcal{H}(y_0,y_1))$ and therefore $y_0\cdot[\zeta]=y_{1}\cdot[\gamma]$. Since $y_{1}\cdot[\gamma]=y_0\cdot[\zeta]$ for \[\zeta\in \mathcal{D}\cap \langle\{1\},W\rangle\subseteq \mathcal{U}\cap \langle \{1\},V\rangle,\]we have $y_1\cdot [\gamma]\in B(y_0\cdot[\alpha],\mathcal{U},V)$.

Since $\pf$ is open by \ref{phopen}, the restriction $B(y_0\cdot[\alpha],\mathcal{U},V)\ra V$ of $\pf$ is a homeomorphism if it is bijective. If $v\in V$, there is a path $\delta\in \mathcal{U}$ such that $\delta(1)=v$ which gives $\pf(y_0\cdot[\delta])=v$. Additionally, if $\delta,\epsilon\in \mathcal{U}\cap \langle \{1\},V\rangle$ such that $\pf(y_0\cdot [\delta])=\delta(1)=\epsilon(1)=\pf(y_0\cdot[\epsilon])$, then $\delta\ast\epsilon^{-1}\in \mathcal{U}\mathcal{U}^{-1}\cap \Omega(X,x_0)\subseteq h^{-1}(\mathcal{H}(y_0))$ and thus $y_{0}\cdot [\delta]=y_0\cdot [\epsilon]$.
\end{proof}
\begin{remark}\emph{
It is worthwhile to note that when $V$ is path connected\[B(y_0\cdot[\alpha],\mathcal{U},V)=\left\{y_0\cdot[\alpha\ast\xi]|\xi\in (\mathcal{P}V)_{\alpha(1)}\right\}.\]Thus if $X$ is locally path connected, the construction of $\txf$ agrees with the widely used construction of coverings (and generalized coverings \cite{FZ07}) of locally path connected spaces.}
\end{remark}
For each $y\in F(x)$, let $y\cdot (\pi X)_{x}=\{y\cdot[\alpha]\in\txf|\alpha\in \pxx\}$, which, by Lemma \ref{basisoftxf}, is open in $\txf$. Since $X$ is path connected, $\pxx$ is path connected and therefore $y\cdot (\pi X)_{x}$ is path connected.
\begin{proposition} \label{pathcomponents}
The path components of $\txf$ are the open sets $y\cdot (\pi X)_{x}$.
\end{proposition}
\begin{proof}
If $y\in y_{1}\cdot (\pi X)_{x_1}\cap y_{2}\cdot (\pi X)_{x_2}$, then $y_1\cdot[\alpha_{1}]=y=y_{2}\cdot [\alpha_2]$ where $\alpha_i(0)=x_i$ and $\alpha_{1}(1)=\alpha_{2}(1)$. We claim that $y_{1}\cdot (\pi X)_{x_1}=y_{2}\cdot (\pi X)_{x_2}$. If $y_{1}\cdot[\beta]\in y_{1}\cdot (\pi X)_{x_1}$ where $\beta(0)=x_1$, then \[y_{1}\cdot[\beta]=y_{1}\cdot[\alpha_{1}][\alpha_{1}^{-1}\ast \beta]=y_{2}\cdot [\alpha_2][\alpha_{1}^{-1}\ast \beta]=y_{2}\cdot [\alpha_2\ast\alpha_{1}^{-1}\ast \beta]\]giving $y_{1}\cdot[\beta]\in y_{2}\cdot (\pi X)_{x_2}$. The other inclusion follows similarly.
\end{proof}
The general idea of the proof of the next proposition is based on that used by Fischer and Zastrow in \cite[Prop. 6.7,6.8]{FZ07} to determine when certain maps have unique path lifting. Recall that if $F$ is a $\spaces$-functor, $y\in F(x)$, and $\alpha\in \Omega(X,x)$, the coset $\mathcal{H}(y)[\alpha]\in \mathcal{H}(y)\bs\pi_{1}^{\tau}(X,x)$ is an open neighborhood of $[\alpha]$ in $\pi_{1}^{\tau}(X,x)$ and $h^{-1}(\mathcal{H}(y)[\alpha])$ is an open neighborhood of $\alpha$ in $\Omega(X,x)$.
\begin{theorem} \label{upl}
If $X$ is locally wep-connected and $F\colon\pi^{\tau} X\ra \set$ is a $\spaces$-functor, then $\pf\colon\txf\ra X$ is a semicovering map.
\end{theorem}
\begin{proof}
By Theorem \ref{contliftingofpf} and Lemma \ref{basisoftxf}, it suffices to show that $\pf$ has unique path lifting. Let $f,g\colon I\ra \txf$ be paths such that $\pf\circ f=\pf\circ g$. We show that either $\{t\in I|f(t)=g(t)\}$ is either empty or $I$. By Proposition \ref{pathcomponents}, we may assume that $f$ and $g$ have image in $y_0\cdot (\pi X)_{x_0}$ where $y_0\in F(x_0)$. Using Corollary \ref{lwtclasses}, we have $f(t)=y_{0}\cdot [\alpha_t]$ and $g(t)=y_{0}\cdot [\beta_t]$ for locally well-targeted paths $\alpha_t,\beta_t\in \pxxo$. The condition $\pf\circ f=\pf\circ g$ means $\alpha_t(1)=\beta_t(1)$ for each $t\in I$ and thus $\alpha_{t}\ast\beta_{t}^{-1}\in \Omega(X,x_0)$. Let $\ell_t=\left[\alpha_{t}\ast\beta_{t}^{-1}\right]$ so that $h^{-1}(\mathcal{H}(y_0)\ell_t)$ is an open neighborhood of $\alpha_{t}\ast\beta_{t}^{-1}$ in $\Omega(X,x_0)$. Since $\alpha_{t}\ast\alpha_{t}^{-1}$ and $\beta_{t}\ast\beta_{t}^{-1}$ are null-homotopic, it is possible to find an open neighborhood $\mathcal{A}_{t}=\bigcap_{j=1}^{n_{t}}\langle K_{n_{t}}^{j},A_{j}^{t}\rangle$ of $\alpha_t$ and $\mathcal{B}_{t}=\bigcap_{i=1}^{m_{t}}\langle K_{m_t}^{i},B_{i}^{t}\rangle$ of $\beta_{t}$ in $\pxxo$ such that \[\left(\left(\mathcal{A}_{t}\mathcal{A}_{t}^{-1}\right)\cup \left(\mathcal{B}_{t}\mathcal{B}_{t}^{-1}\right)\right)\cap \Omega(X,x_0)\subseteq h^{-1}(\mathcal{H}(y_0))\text{  and  }\mathcal{A}_{t}\mathcal{B}_{t}^{-1}\cap \Omega(X,x_0)\subseteq h^{-1}(\mathcal{H}(y_0)\ell_{t})\]Since $\alpha_t,\beta_t$ are locally well-targeted, there is an open neighborhood $U_t\subseteq A_{n_t}^{t}$ of $\alpha_t(1)$ (resp. $V_t\subseteq B_{m_t}^{t}$ of $\beta_t(1)$) such that for each $u\in U_t$ (resp. $v\in V_t$), there is a {\wt} path $\delta\in \mathcal{A}_{t}$ with $\delta(1)=u$ (resp. $\gamma\in \mathcal{B}_{t}$ with $\gamma(1)=v$). According to Lemma \ref{basisoftxf}, for each $t\in I$,\[B(y_{0}\cdot[\alpha_t],\mathcal{A}_t,U_t)\text{  and  }B(y_{0}\cdot[\beta_t],\mathcal{B}_t,V_t)\]are open neighborhoods of $y_{0}\cdot[\alpha_t]$ and $y_{0}\cdot[\beta_t]$ in $y_0\cdot (\pi X)_{x_0}$ respectively. Suppose there are $r,s\in I$ such that $y_{0}\cdot[\alpha_r]\neq y_{0}\cdot[\beta_r]$ and $y_{0}\cdot[\alpha_s]=y_{0}\cdot[\beta_s]$. Without loss of generality, we assume $r<s$. Let $z$ be the greatest lower bound of $A=\{t\in [r,s]|y_{0}\cdot[\alpha_{t}]=y_{0}\cdot[\beta_t]\}=\{t\in [r,s]|[\alpha_{t}\ast\beta_{t}^{-1}]\in \mathcal{H}(y_0)\}$. Since $f$ and $g$ are continuous, there is an $\epsilon>0$ such that $y_{0}\cdot[\alpha_t]\in B(y_{0}\cdot[\alpha_z],\mathcal{A}_z,U_z)$ and $y_{0}\cdot[\beta_t]\in B(y_{0}\cdot[\beta_z],\mathcal{B}_z,V_z)$ for all $t\in (z-\epsilon,z+\epsilon)\cap[0,1]$. We consider two cases:

(1) If $z\in A$ (equivalently $[\alpha_z\ast\beta_{z}^{-1}]\in \mathcal{H}(y_0)$), then $r<z\leq s$ and $\mathcal{H}(y_0)\ell_z=\mathcal{H}(y_0)$. Pick any $t_0\in (r,z)\cap (z-\epsilon,z)$. We have \[y_{0}\cdot[\alpha_{t_0}]\in B(y_{0}\cdot[\alpha_z],\mathcal{A}_z,U_z)\text{ and }y_{0}\cdot[\beta_{t_0}]\in B(y_{0}\cdot[\beta_z],\mathcal{B}_z,V_z)\] and therefore $y_{0}\cdot[\alpha_{t_0}]=y_{0}\cdot[\zetaup]$ for $\zetaup\in \mathcal{A}_z$ and $y_{0}\cdot[\beta_{t_0}]=y_{0}\cdot[\etaup]$ for $\etaup\in \mathcal{B}_z$. Since $\zetaup(1)=\alpha_{t_0}(1)=\beta_{t_0}(1)=\etaup(1)$, we have \[\zetaup\ast\etaup^{-1}\in \mathcal{A}_{z}\mathcal{B}_{z}^{-1}\cap \Omega(X,x_0)\subseteq h^{-1}(\mathcal{H}(y_0)\ell_{z})=h^{-1}(\mathcal{H}(y_0))\]and $y_{0}\cdot[\alpha_{t_0}]=y_{0}\cdot[\zetaup]=y_{0}\cdot[\etaup]=y_{0}\cdot[\beta_{t_0}]$. But $t_0<z$ and $t_0\in A$ contradicting that $z$ is a lower bound for $A$.

(2) If $z\notin A$ (equivalently $[\alpha_z\ast\beta_{z}^{-1}]\notin \mathcal{H}(y_0)$), then $r\leq z< s$ and $\mathcal{H}(y_0)\ell_z \cap \mathcal{H}(y_0)=\emptyset$. Pick any $t_0\in (z,s)\cap (z,z+\epsilon)$ so that, again, $y_{0}\cdot[\alpha_{t_0}]=y_{0}\cdot[\zetaup]$ for $\zetaup\in \mathcal{A}_z$ and $y_{0}\cdot[\beta_{t_0}]=y_{0}\cdot[\etaup]$ for $\etaup\in \mathcal{B}_z$. If $y_{0}\cdot[\alpha_{t_0}]=y_{0}\cdot[\beta_{t_0}]$, then $[\zetaup\ast\etaup^{-1}]\in \mathcal{H}(y_0)$. But this cannot occur since \[\zetaup\ast\etaup^{-1}\in \mathcal{A}_{z}\mathcal{B}_{z}^{-1}\cap \Omega(X,x_0)\subseteq h^{-1}(\mathcal{H}(y_0)\ell_z)\]and $\mathcal{H}(y_0)\ell_z\cap \mathcal{H}(y_0)=\emptyset$. Thus $y_{0}\cdot[\alpha_t]\neq y_{0}\cdot[\beta_{t}]$ for each $t\in [z,s)\cap [z,z+\epsilon)$. That any $y\in (z,s)\cap (z,z+\epsilon)$ is a lower bound for $A$ which is greater than $z$ is a contradiction.
\end{proof}
\subsection{Proof of Theorem \ref{classificationI}}

To complete the proof of Theorem \ref{classificationI}, we define the inverse of monodromy $\mathscr{M}\colon\scovx\ra \mathbf{TopFunc}(\pi^{\tau}X,\set)$ (for locally wep-connected $X$) as the functor $\mathscr{S}\colon\mathbf{TopFunc}(\pi^{\tau}X,\set)\ra \scovx$ given by $\mathscr{S}(F)=\pf$ on objects and $\mathscr{S}(\eta\colon F\ra G)=p_{F,G}$ (as in Proposition \ref{morphisms}) on morphisms. Note that in the case that $F(x)=\emptyset$ for each $X$, $\mathscr{S}(F)$ is the empty semicovering $\emptyset\ra X$. It is clear that $\mathscr{M}(\mathscr{S}(F))=\mathscr{M}(\pf)=F$ and $\mathscr{M}(\mathscr{S}(\eta))=\mathscr{M}(p_{F,G})=\eta$ for a $\spaces$-natural transformation $\eta\colon F\ra G$ since $p_{F,G}$ is given by $p_{F,G}(y)=\eta_{x}(y)$ for $y\in F(x)=p_{F}^{-1}(x)$. Thus $\mathscr{M}\circ \mathscr{S}=Id$.

Suppose $\pyx$ is a semicovering of $X$ and $K=\mathscr{M}p$. We have $\tilde{X}_{K}=\bigcup_{x\in X}K(x)=\bigcup_{x\in X}p^{-1}(x)=Y$ as sets and thus $\mathscr{S}(K)=p_{K}=p$ as functions. To see that the topologies of $Y$ and $\txf$ agree consider the following diagram:\[\xymatrix{ \coprod_{x\in X}p^{-1}(x)\times \pxx \ar[d]_-{\coprod_{y}L_p} \ar[rr]^-{\Theta_{K}} && \tilde{X}_{K} \ar[d]^-{id} \\ \coprod_{y\in Y}\pyy \ar[rr]_-{\coprod_{y} ev_{1}} && Y}\] The left vertical map takes $(y,\alpha)\in p^{-1}(x)\times \pxx$ to the lift $\tilde{\alpha}_{y}\in \pyy$ and is a homeomorphism since $p$ has continuous lifting of paths. The bottom horizontal map is evaluation at $1$ on each summand and is quotient by Propositions \ref{liftingproperties} and \ref{quotienteval}. Since $\Theta_{K}$ is quotient by definition, $id\colon\tilde{X}_{K}\ra Y$ is continuous and open. One could equally have chosen to apply Lemma \ref{generalliftinglemma} to draw this conclusion.

Finally, suppose $p'\colon Y'\ra X$ is another semicovering of $X$ and $f\colon Y\ra Y'$ is a map such that $p'\circ f=p$. Let $K=\mathscr{M}p$, $K'=\mathscr{M}p$, and $\eta=\mathscr{M}f\colon K\ra K'$ so that $\eta_{x}\colon p^{-1}(x)\ra (p')^{-1}(x)$ is the restriction of $f$. The covering morphism $\mathscr{S}(\eta)=p_{K,K'}\colon Y\ra Y'$ is given by $p_{K,K'}(y)=\eta_{x}(y)=f(y)$. Thus $\mathscr{S}\circ \mathscr{M}=Id$.
\subsection{Open covering morphisms} 
It is well-known that the category of covering morphisms $\mathbf{CovMor}(\sG)$ of a connected groupoid $\sG$ (connected in the categorical sense that $\sG(x_1,x_2)\neq\emptyset$ for all $x_1,x_2\in Ob(\sG)$) is naturally equivalent to $\mathbf{Func}(\sG,\set)$ which is often referred to as the category of representations of $\sG$ \cite[Prop. 30]{Hig71}. In particular, a covering morphism $F\colon\sH\ra \sG$ corresponds to the functor $\mathscr{R}F\colon\sG\ra \set$ given by $\mathscr{R}F(x)=Ob(F)^{-1}(x)$ for $x\in Ob(\sG)$ and for $g\in \sG(x_1,x_2)$, $\mathscr{R}F(g)$ is the function $Ob(F)^{-1}(x_1)\ra Ob(F)^{-1}(x_2)$, $y\mapsto t_{\sH}(\tilde{g}_{y})$ where $t_{\sH}\colon\sH\ra Ob(\sH)$ is the target map of $\sH$. 

For a connected (in the categorical sense) $\spaces$-groupoid $\sG$, let $\mathbf{OCovMor}(\sG)$ be the category of open covering morphisms of $\sG$, that is, open $\spaces$-functors $\sH\ra \sG$ whose underlying functors are covering morphisms. In this enriched setting the equivalence $\mathbf{CovMor}(\sG)\simeq \mathbf{Func}(\sG,\set)$ restricts to an equivalence \[\mathscr{R}\colon\mathbf{OCovMor}(\sG)\ra \mathbf{TopFunc}(\sG,\set).\] This is straightforward given the observation that a subbasis set for the topology of \[\set\left(Ob(F)^{-1}(x_1), Ob(F)^{-1}(x_2)\right)\] is of the form $\langle \{y_1\},\{y_2\}\rangle$ and $\mathscr{R}F\colon\sG(x_1,x_2)\ra \set(Ob(F)^{-1}(x_1),Ob(F)^{-1}(x_2))$ is continuous precisely when \[\mathscr{R}F^{-1}\left(\langle \{y_1\},\{y_2\}\rangle\right)=\{g\in \sG(x_1,x_2)|t_{\sH}(\tilde{g}_{y_1})=y_2\}=Im\left(F\colon\sH(y_1,y_2)\ra \sG(x_1,x_2)\right)\] is open in $\sG(x_1,x_2)$. 

By Theorem \ref{embedding} for any space $X$, there is a functor $\pi^{\tau}_{!}\colon\scovx\ra \mathbf{OCovMor}(\pi^{\tau}X)$, $p\mapsto\pi^{\tau}p$. Notice that $\mathscr{R}\circ \pi^{\tau}_{!}=\mathscr{M}$ is given by monodromy. Thus for locally-wep connected $X$, we obtain an alternative classification of semicoverings in terms of open coverings morphisms.
\begin{theorem}
For a locally wep-connected space $X$, $\pi^{\tau}_{!}\colon\scovx\ra \mathbf{OCovMor}(\pi^{\tau}X)$ is a natural equivalence of categories.
\end{theorem}
\begin{corollary}
If $X$ is locally wep-connected and $F\colon\sH\ra \pi X$ is a covering morphism such that each connected component of $\sH$ contains an object $y$ such that $F(\sH(y))$ is an open subgroup of $\pi_{1}^{\tau}(X,F(y))$, then the topology of $X$ lifts to a topology on $Ob(\sH)$ such that $Ob(F)$ is a semicovering map and $\sH\cong \pi^{\tau} Ob(\sH)$ as $\spaces$-groupoids when $\sH(y_1,y_2)$ is viewed as a subspace of $\pi^{\tau}(F(y_1),F(y_2))$.
\end{corollary}
\subsection{Semicoverings and $G$-sets}
Let $G$ be a topological group and $F$ be a set with the discrete topology. A right action of $G$ on $F$ is a continuous group action $F\times G\ra F$. For each $x\in F$, the restriction $G\ra xG=\{xg\in F|g\in G\}$ of the action induces a continuous bijection from the quotient right coset space $H\bs G$ (where $H$ is the stabilizer at $x$) to $xG$. Thus continuity of the action of $G$ on discrete $F$ is equivalent to all stabilizer subgroups being open in $G$. A \textit{right} $\textit{G}$-\textit{set} is a discrete set $F$ with a right action of $G$. This agrees with the usual notion of right $G$-set when $G$ is discrete. The category of right $G$-sets and $G$-maps $\psi\colon F_1\ra F_2$ is denoted $G\mathbf{Set}$. Note that $G\mathbf{Set}$ is naturally equivalent to $\mathbf{TopFunc}(G,\set)$ when $G$ is viewed as a $\spaces$-groupoid with a single object.

A right $G$-set $F$ is \textit{transitive} if $F=xG$ for some $x\in F$. In this case, $F$ may be identified in $G\mathbf{Set}$ with the discrete right coset space $H\bs G=\{Hg|g\in G\}$ where again $H$ is the stabilizer at $x$. Define the \textit{orbit category} of the topological group $G$ to be the full subcategory $\mathscr{O}_{G}$ of $G\mathbf{Set}$ generated by transitive right $G$-sets. Note that $\mathscr{O}_{G}$ is equivalent to the full subcategory of $G\mathbf{Set}$ generated by $G$-sets $H\bs G$ where $H$ is an open subgroup of $G$. The well-known theory for discrete groups extends to the non-discrete case giving that the objects of $\mathscr{O}_{G}$ correspond to open subgroups of $G$ and isomorphism classes correspond to conjugacy classes of open subgroups of $G$.
\begin{theorem} \label{ClassificationII} Let $X$ be a locally wep-connected space and $x_0\in X$.
\begin{enumerate} 
\item There is an equivalence $\scovx\ra \pi_{1}^{\tau}(X,x_0)\mathbf{Set}$ of categories taking a semicovering $\pyx$ to the fiber $p^{-1}(x_0)$ with action $p^{-1}(x_0)\times \pi_{1}^{\tau}(X,x_0)\ra p^{-1}(x_0)$, $(y,[\alpha])\mapsto y\cdot [\alpha]$ is an equivalence of categories.
\item The equivalence in 1. restricts to an equivalence of categories $\cscovx\simeq\mathscr{O}_{\pi_{1}^{\tau}(X,x_0)}$.
\end{enumerate}
\end{theorem}
\begin{proof}
1. The inclusion $\pi_{1}^{\tau}(X,x_0)\to \pi^{\tau} X$ is a $\spaces$-equivalence since $X$ is path connected and translations in $\pit X$ are continuous. By Theorem \ref{classificationI}, Monodromy gives the first isomorphism in \[\scovx\cong\mathbf{TopFunc}(\pi^{\tau}X,\set)\simeq \mathbf{TopFunc}(\pi_{1}^{\tau}(X,x_0),\set)\simeq \pi_{1}^{\tau}(X,x_0)\mathbf{Set}.\]
2. The equivalence from 1. restricts to an equivalence $\cscovx\simeq \mathscr{O}_{\pi_{1}^{\tau}(X,x_0)}$ due to the fact that given a semicovering $\pyx$, $Y$ is path connected if and only if the action $p^{-1}(x_0)\times \pi_{1}^{\tau}(X,x_0)\ra p^{-1}(x_0)$ is transitive.
\end{proof}
\begin{corollary}
Let $X$ be a locally wep-connected space and $x_0\in X$. There is a Galois correspondence between the equivalence classes of connected semicoverings of $X$ and conjugacy classes of open subgroups of $\pi_{1}^{\tau}(X,x_0)$.
\end{corollary}
\begin{corollary}
If $X$ is {\estc} and $x_0\in X$, then $X$ has a universal semicovering if and only if there is an open subgroup $S$ in $\pi_{1}^{\tau}(X,x_0)$ such that for any other open subgroup $H$ of $\pi_{1}^{\tau}(X,x_0)$, there is a $g\in \pi_{1}^{\tau}(X,x_0)$ such that $gSg^{-1}\subseteq H$.
\end{corollary}
\end{document}